\documentclass[11pt, a4paper, english]{amsart}
\usepackage[top=2cm, bottom=2cm, left=2.2cm, right=2.2cm, twoside=false]{geometry}
\usepackage[english]{babel} 
\usepackage[utf8]{inputenc}
\usepackage{color}
\usepackage{amstext}
\usepackage{ifpdf}
\usepackage{amsmath}
\usepackage{amsthm}
\usepackage{amssymb}
\usepackage{amscd}
\usepackage{enumerate}
\usepackage[normalem]{ulem}
\usepackage[all]{xypic}

\newcommand{\comm}[1]{}
\numberwithin{equation}{section}
\numberwithin{figure}{section}
\theoremstyle{plain}
\newtheorem*{thm*}{\protect\theoremname}
\newtheorem{thm}{\protect\theoremname}[section]
\newtheorem{lem}[thm]{\protect\lemmaname}
\newtheorem*{prob}{\protect\problemname}
\newtheorem{cor}[thm]{\protect\corollaryname}
\newtheorem{prop}[thm]{\protect\propositionname}

\theoremstyle{definition}
\newtheorem{defn}[thm]{\protect\definitionname}
\newtheorem{example}[thm]{\protect\examplename}
\newtheorem{Rem}[thm]{\protect\Remarkname}
\newtheorem{nota}[thm]{\protect\notationname}

\theoremstyle{remark}
\newtheorem{rem}[thm]{\protect\remarkname}

\newtheorem*{claim*}{Claim}

\newcommand{\caL}{\mathcal{L}}
\newcommand{\cC}{\mathcal{C}}
\newcommand{\cE}{\mathcal{E}}
\newcommand{\cS}{\mathcal{S}}

\newcommand{\cI}{\mathcal{I}}
\newcommand{\cJ}{\mathcal{J}}
\newcommand{\cK}{\mathcal{K}}

\newcommand{\cO}{\mathcal{O}}
\newcommand{\cF}{\mathcal{F}}

\newcommand{\cN}{\mathcal{N}}
\newcommand{\cV}{\mathcal{V}}
\newcommand{\fm}{\mathfrak{m}}

\renewcommand{\P}{\mathbb{P}}

\newcommand{\F}{\mathbb{F}}

\newcommand{\Q}{\mathbb{Q}}
\newcommand{\Z}{\mathbb{Z}}
\newcommand{\C}{\mathbb{C}}
\newcommand{\A}{\mathbb{A}}
\newcommand{\N}{\mathbb{N}}
\renewcommand{\:}{\colon}
\newcommand{\8}{\infty}
\renewcommand{\to}{\rightarrow}
\newcommand{\map}{\dashrightarrow}
\renewcommand{\,}{\ }

\makeatletter \def\subsection{\@startsection{subsection}{3}%
  \z@{.5\linespacing\@plus.7\linespacing}{.5\linespacing}%
  {\normalfont\itshape\bfseries}} \makeatother  %could use \normalfont\bfseries too

\usepackage{xypic}
\usepackage{mathrsfs, setspace, fancyhdr, times, bm}

\providecommand{\definitionname}{Definition}
\providecommand{\examplename}{Example}
\providecommand{\lemmaname}{Lemma}
\providecommand{\corollaryname}{Corollary}
\providecommand{\propositionname}{Proposition}
\providecommand{\remarkname}{Remark}
\providecommand{\Remarkname}{Remark}
\providecommand{\theoremname}{Theorem}
\providecommand{\problemname}{Problem}
\providecommand{\notationname}{Notation}

\usepackage{xcolor}

\def\Bs{\operatorname{Bs}}

\def\dom{\operatorname{dom}}
\def\NE{\operatorname{NE}}
\def\Supp{\operatorname{Supp}}
\def\Discrep{\operatorname{Discrep}}
\def\inf{\operatorname{inf}}
\def\depth{\operatorname{depth}}

\def\pr{\operatorname{pr}}
\def\pp{\operatorname{p}}
\def\Exc{\operatorname{Exc}}
\def\Proj{\operatorname{Proj}}

\def\Sym{\operatorname{Sym}}

\def\Frac{\operatorname{Frac}}

\def\Spec{\operatorname{Spec}}

\def\Sing{\operatorname{Sing}}
\def\Pic{\operatorname{Pic}}
\def\NS{\operatorname{NS}}
\def\Cl{\operatorname{Cl}}
\def\Cl{\operatorname{Cl}}

\def\div{\operatorname{div}}
\def\ev{\operatorname{ev}}

\def\Bl{\operatorname{Bl}}

\def\Gr{\operatorname{Gr}}
\def\Coker{\operatorname{Coker}}
\def\Im{\operatorname{Im}}
\def\HHom{\operatorname{\mathcal{H}om}}

\def\codim{\operatorname{codim}}
\def\redd{\mathrm{red}}

\def\reg{\mathrm{reg}}
\def\iff{\ \Leftrightarrow}
\def\mono{\hookrightarrow}

\newcommand{\ov}{\overline}

  \usepackage[colorlinks=true, citecolor=teal,linkcolor=blue,  pagebackref=false]{hyperref} 
\begin{document}

\title[Completions of affine spaces into Mori fiber spaces]{Completions of affine spaces \\ into Mori fiber spaces with non-rational fibers}

\author{Adrien Dubouloz}
\address{IMB UMR5584, CNRS, Univ. Bourgogne Franche-Comt\'e, EIPHI Graduate School  ANR-17-EURE-0002, F-21000 Dijon, France.}
\email{adrien.dubouloz@u-bourgogne.fr}

\author{Takashi Kishimoto}
\address{Department of Mathematics, Faculty of Science, Saitama University, Saitama 338-8570, Japan}
\email{tkishimo@rimath.saitama-u.ac.jp}

\author{Karol Palka}
\address{Institute of Mathematics, Polish Academy of Sciences, \'{S}niadeckich 8, 00-656 Warsaw, Poland}
\email{palka@impan.pl}

\thanks{The authors were supported by the National Science Centre, Poland, grant number 2015/18/E/ST1/00. The first author was partially supported by the French ANR project ”FIBALGA” ANR-18-CE40-0003. The second author was partially supported by JSPS KAKENHI, grant number JP19K03395. For the purpose of Open Access, the authors have applied a CC-BY public copyright license to any Author Accepted Manuscript version arising from this submission.}

\begin{abstract}
We describe a method to construct completions of affine spaces into total spaces of $\Q$-factorial terminal Mori fiber spaces over the projective line. As an application we provide families of examples with non-rational, birationally rigid and non-stably rational general fibers.
\end{abstract}

 % % MSC2020 14R10, 14E30, 14E08
 
\maketitle

\section{Introduction}

We work with complex algebraic varieties. A \emph{completion} of a given variety $U$ is a complete variety containing a Zariski open subset isomorphic to $U$. In this article we consider the problem of describing minimal completions of affine spaces $\A^n$. Since the Kodaira dimension of $\A^n$ is negative, a natural way to define minimality in this context is to require the completion to be the total space of a Mori fiber space. Indeed, by \cite[Corollary 1.3.3]{BCHM} such varieties come as outputs of the Minimal Model Program applied to smooth projective varieties of negative Kodaira dimension. We call them \emph{Mori fiber completions}. From the viewpoint of the Minimal Model Program it is also natural to consider not only smooth but also mildly singular varieties and their completions, namely those which are $\Q$-factorial and have terminal singularities.   

If $\pi\:V \to B$ is a Mori fiber completion of $\A^n$ then $V$ is rational and the base variety $B$ is unirational. The case when $B$ is a point is especially important, because then $V$ is a Fano variety of Picard rank one. Fano varieties and their rationality are objects of intensive studies, see \cite{Iskovskikh_Prokhorov-Fano_varieties}, \cite{Kollar-Smith-Corti_Nearly_rational}. Classifying smooth Fano completions of $\A^n$ of Picard rank one is the projective version of the celebrated problem of finding minimal analytic completions of complex affine spaces raised by Hirzebruch \cite[Problem 27]{Hirzebruch_problems} and studied by many authors. In case $n=1,2$ there is only $\P^n$. The first difficult case, $n=3$, was completed in a series of papers \cite{H3a, H3b, Pr, Fu93}, see also \cite{Kis05} for partial classification results concerning completions of $\A^3$ into smooth Fano threefolds with Picard rank two. For $n=4$ there are some partial results \cite{Pr93, PZ-genus10_4folds_and_A4}. 

In this article we focus on the situation where the base $B$ is a curve, hence is isomorphic to the projective line $\P^1$. Simple examples of Mori fiber completions of $\A^n$ of this type are given by locally trivial $\P^{n-1}$-bundles over $\P^1$. Another series of examples can be constructed by taking the product of $\P^1$ with any $\Q$-factorial terminal Fano variety of Picard rank one which is a completion of $\A^{n-1}$, see Example \ref{ex:products}. These examples are special in the sense that general fibers are completions of $\A^{n-1}$. In general, the following interesting problem arises:

\begin{prob}
Let $\pi\: V\to \P^1$ be a Mori fiber completion of $\A^n$. What can be said about the geometry of general fibers of $\pi$? 
\end{prob}

A basic observation is that a general fiber of $\pi$ is a Fano variety of dimension $n-1$ with terminal singularities. The property of being a Mori fiber space implies in particular that the generic fiber of $\pi$ has Picard rank one over the function field of $\P^1$. On the contrary, general fibers do not necessarily have Picard rank one, as can be seen for instance for del Pezzo fibrations of some threefolds completing $\A^3$, see Example \ref{ex:dP_fibr}. Still, it is a restrictive condition for a Fano variety to be a general fiber of a Mori fiber space, see \cite{CFST-Fano_fibers, CFST-Fano_fibers2}. 
For $n=2, 3$ a general fiber of $\pi$, being terminal and Fano, is either $\P^1$ or a smooth del Pezzo surface, hence is a completion of $\A^{2}$. In contrast, our first result implies that in higher dimensions general fibers of Mori fiber completions of $\A^n$ can be very far from being rational.

\begin{thm}\label{thm2}
Let $H$ be a hyperplane in $\P^n$, $n\geq 2$. For every integral hypersurface $F\subseteq \P^n$ of degree $d\leq n$ such that $F\cap H$  is irreducible and contained in the smooth locus of $F$ there exists a Mori fiber completion $\pi\:V\to \P^1$ of the affine $n$-space $\A^n \cong \P^n \setminus H$ such that all hypersurfaces other than $dH$ in the pencil of divisors $\langle F,dH\rangle$ generated by $F$ and $dH$ appear as fibers of $\pi$.
\end{thm}

By Bertini's theorem a general member of a pencil as in Theorem \ref{thm2} is smooth. For $(n,d)=(4,3)$ it is a smooth cubic threefold, hence is unirational but not rational \cite{CM72}. For $d=n\geq 4$ a general member is birationally super-rigid (see Definition \ref{def:rigid}) by \cite{deF13, IsMa71,Pu98}. This gives the following corollary. 

\begin{cor}\label{cor:completions_of_A4}
For every $n\geq 4$ there exists a Mori fiber completion of $\A^n$ over $\P^1$ whose general fibers are smooth birationally super-rigid Fano varieties of Picard rank one.
\end{cor}

Another corollary concerns completions of polynomial morphisms $f\:\A^n\to \A^1$ of low degree into Mori fiber spaces over $\P^1$. 

\begin{cor}\label{cor:polynomial} Assume that $f\: \A^n\to \A^1$ is a morphism given by a polynomial of total degree at most $n$ and that in the natural open embedding $\A^n\subseteq \P^n$ the intersection of the closure of the zero locus of $f$ with $\P^n\setminus \A^n$ is smooth. Then $f\:\A^n\to \A^1$ can be completed into a Mori fiber space over $\P^1$. 
\end{cor} 

General fibers of polynomial morphisms $f\:\A^n\to \A^1$ as above of degree $\deg f\leq n-1$ are smooth affine Fano varieties in the sense of \cite{CDP17} (see Definition \ref{def:Affine-Fano}). We note that for $n\geq 6$ and $\deg f=n-1$ general fibers of the Mori fiber space  $\pi\:V\to\P^1$ are then completions of so-called \emph{super-rigid affine Fano varieties} (general fibers of $f$), see Definition \ref{def:affine_rigid} and Example \ref{ex:affine_super-rigid}.

\medskip

We obtain even more families of possible general fibers by considering singular ambient spaces instead of $\P^n$. For a definition of quasi-smooth hypersurfaces in weighted projective spaces, see Section \ref{sec:Cone_construction}.

\begin{thm}\label{thm1}
Let $n\geq 4$ and let $\P=\P(1,a_1,\ldots,a_n)$ be a weighted projective space for some positive integers $a_1,\ldots, a_n$, such that the description of the hyperplane $H=\P(a_1,\ldots,a_n)$  is well-formed. Then for every quasi-smooth terminal hypersurface $F\neq H$ of $\P$ of degree $d\leq a_1+\ldots+a_n$ there exists a Mori fiber completion  $\pi\:V\to \P^1$ of $\A^n \cong \P \setminus H$ such that all hypersurfaces other than $dH$ in the pencil $\langle F,dH\rangle$ generated by $F$ and $dH$ appear as fibers of $\pi$.
\end{thm}

For $n=4$ we get a class of Mori fiber completions of $\A^4$ over $\P^1$ whose general fibers are quasi-smooth terminal weighted Fano threefold hypersurfaces in the 95 families of Fletcher and Reid \cite{Fl00,Reid-canonical_3folds}, see Corollary \ref{cor:95families_bir_rigid_fibers}. By \cite[Main Theorem]{CP16}, see also \cite{CPR00}, all such threefolds are birationally rigid and some of them are even known to be birationally super-rigid \cite[Theorem 1.1.10]{CP16}. In a similar vein, we deduce from \cite{Okada-stable_rationality_of_Fano_3fold_hypersurfaces} the following result, see Example \ref{okada}.

\begin{cor}\label{cor:completions_of_A4_Okada} There exists a Mori fiber completion of $\A^4$ over $\P^1$ whose very general fibers are Fano varieties of Picard rank one which are not stably rational.
\end{cor}

We now briefly describe our approach. A natural way to obtain completions of a given quasi-projective variety $U$ with $\Q$-factorial terminal singularities ($\A^n$ in particular) is to find some normal projective completion whose singularities are not worse than those of $U$, and then to run a Minimal Model Program on it. If $U$ is smooth then we may take a smooth completion using resolution of singularities. But in general, finding appropriate completions from which to run the program is already a nontrivial task. Moreover, each step, whether it is a divisorial contraction or a flip, may change the isomorphism type of the image of $U$. Preventing this to happen is one of the key problems. To gain more control over the successive steps of the program we study completions which are resolutions of specific pencils of divisors on terminal Fano varieties, namely of those pencils whose general members are Fano varieties of Picard rank one with terminal singularities. We call them \emph{terminal rank one Fano pencils}. This assumption on the one hand allows to find a completion with mild singularities and on the other hand, it gives a chance to analyze the MMP runs in a more detail. We introduce the notion of a ``compatible thrifty resolution" of such pencils, characterized essentially by the property that it keeps the isomorphism type of general members unchanged. We give sufficient criteria for the existence of compatible thrifty resolutions and we show that terminal rank one Fano pencils which admit such resolutions yield interesting Mori fiber spaces over $\P^1$. 

The article is organized as follows: In Section \ref{sec:prelim} we recall basic notions concerning varieties and their singularities in the framework of the Minimal Model Program. Section \ref{sec:Weil-Pencils} reviews properties of pencils of Weil divisors on normal varieties. Section \ref{sec:pencils_and_resolutions} is devoted to the study of terminal rank one Fano pencils, their resolutions and the outputs of relative MMP's ran from these. In Section \ref{sec:H-special} we consider a class of pencils on Fano varieties with class groups $\Z$. It provides a big supply of terminal rank one Fano pencils admitting compatible thrifty resolutions. Applications to the construction of Mori fiber completions of $\A^n$ over $\P^1$ are given in Section \ref{sec:Affine-Spaces}. This section contains proofs of Theorems \ref{thm2} and \ref{thm1} and a series of examples.

\medskip \textit{Acknowledgements.} We thank the Institute of Mathematics of Burgundy, the Saitama University and the Institute of Mathematics of the Polish Academy of Sciences for excellent working conditions. We thank the referees for their suggestions of improvements of the text.

\tableofcontents

\section{Preliminaries}\label{sec:prelim}

We summarize basic notions concerning varieties and their singularities in the framework of the Minimal Model Program which are used in the article. 

We use the following standard terminology: The \emph{domain of the definition} of a dominant rational map $f\:X\map Y$ between algebraic varieties is the largest open subset $\dom(f)$ of $X$ on which $f$ is represented by a morphism. Its complement is called the \emph{indeterminacy locus} of $f$. The \emph{exceptional locus}  $\Exc(f)$ of a proper birational morphism $f\:X\to Y$ is the pre-image of the indeterminacy locus of the birational map $f^{-1}\:Y\map X$. A \emph{resolution} (of indeterminacy) of a rational map $f\:X\map Y$  is a proper birational morphism $\tau\:X'\to X$ such that $f\circ\tau\:X'\map Y$ is a morphism.  
 
\subsection{Singularities in the context of MMP} \label{ssec:singularities}

Let $X$ be a normal variety and let $j\:X_\reg\hookrightarrow X$ be the embedding of the smooth locus. The induced restriction on Picard groups $j^*\:\Pic(X)\to \Pic(X_\reg)$ is injective and the restriction on class groups $\Cl(X)\to\Cl(X_\reg)\cong \Pic(X_\reg)$ is an isomorphism. This gives a natural injection $\Pic(X)\to \Cl(X)$; see \cite[Corollaire 21.6.10]{EGAIV-4}. A canonical divisor of $X$ is a Weil divisor $K_X$ on $X$ whose class in $\Pic(X_\reg)$ is the class of the canonical invertible sheaf $\det(\Omega^1_{X_\reg})$ of $X_\reg$. A Weil divisor on $X$ is called \emph{$\Q$-Cartier} if it has a positive multiple which is Cartier. We say that $X$ is \emph{$\Q$-factorial} if every Weil divisor on $X$ is $\Q$-Cartier.

\medskip 
We now recall some basic facts about singularities of pairs. We refer the reader to \cite[Chapter 2]{KollarMori-Bir_geometry} for details. A \emph{log pair} $(X,D)$ consists of a normal variety $X$ and a Weil $\Q$-divisor $D=\sum d_iD_i$ on it, whose coefficients $d_i$ belong to $[0,1]\cap \Q$, and such that the divisor $K_X+D$ is $\Q$-Cartier. Given a proper birational morphism $f\: Y\to X$ from a normal variety $Y$, we denote by $\cE(f)$ the set of prime divisors $E$ on $Y$ contained in the exceptional locus $\Exc(f)$ of $f$. We call the image of $E$ the \emph{center of $E$ on $X$}.

Given a log pair $(X,D)$ and a birational proper morphism from a normal variety $f\:Y\to X$ we have a linear equivalence of $\Q$-divisors $$K_Y+f_*^{-1}D \sim f^*(K_X+D)+\sum_{E\in \cE(f)}a_{X,D}(E) E,$$ where $a_{X,D}(E) \in \Q$. The number $a_{X,D}(E)$ is called the \emph{discrepancy} of $E$ with respect to $(X,D)$. It does not depend on $Y$ in the sense that if $Y'\to Y$ is a proper birational morphism and $E'$ is the proper transform of $E$ on $Y'$ then $a_{X,D}(E)=a_{X,D}(E')$. The \emph{discrepancy} of the log pair $(X,D)$ is defined as
\begin{equation}\label{eq:discr}
\Discrep (X,D)=\underset{f,E\in \cE(f)}{\inf} a_{X,D}(E),
\end{equation}
where the infimum is taken over all $f\:Y\to X$ as above and all $E\in \cE(f)$.  

A log pair $(X,D)$ is \emph{terminal} if $\Discrep (X,D)>0$, \emph{purely log terminal (plt)} if $\Discrep (X,D)> -1$ and \emph{Kawamata log terminal (klt)} if it is plt and $\lfloor D\rfloor=0$. A normal variety $X$ is \emph{terminal} (respectively klt) if the log pair $(X,0)$ is terminal (respectively klt). The property of being terminal for a variety $X$ and being klt for any log pair $(X,D)$ can be verified by computing the infimum \eqref{eq:discr} on a single proper birational morphism $f\:Y\to X$ which is a log resolution of the log pair $(X,D)$, that is, for which $Y$ is smooth and $\Exc(f)\cup f^{-1}_*(D)$ is a divisor with simple normal crossings, see \cite[Corollaries 2.12,~2.13]{Kollar-Singularities_of_MMP}.  

Given a normal variety $X$ and a closed subset $Z$ of it, we say that $X$ is \emph{terminal (respectively klt) in a neighborhood of $Z$} if there exists an open neighborhood $U$ of $Z$ which is terminal (respectively klt).  Finally, we say that a log pair $(X,S)$, where $S$ is a prime Weil divisor on $X$, is \emph{plt in a neighborhood of $S$} if there exists an open neighborhood $U$ of $S$ such that the log pair $(U,S)$ is plt. Equivalently, if for every log resolution $f\:Y\to X$ of $(X,S)$ every exceptional divisor $E$ of $f$ whose center on $X$ intersects $S$ has discrepancy $a_{X,S}(E)>-1$. 

We will need the following known result.

\begin{lem} \label{lem:plt-to-terminal}Let $(X,S)$ be a log pair such that $S$ is a 
prime Cartier divisor. Assume that $(X,S)$ is plt in a neighborhood of $S$. Then $X$ is terminal in a neighborhood of $S$ if and only if for every log resolution $f\:Y\to X$ of $(X,S)$ every exceptional divisor $E$ of $f$ whose center on $X$ meets $S$ but is not contained in $S$ has positive discrepancy.
\end{lem}   

\begin{proof}
Let $f\:Y\to X$ be a log resolution of $(X,S)$. Replacing $X$ by an open neighborhood of $S$, if necessary, we can assume that the center on $X$ of every exceptional divisor $E_i$ of $f$ meets $S$. Since $S$ is a prime Cartier divisor, we have $f^*S=f_*^{-1}(S)+\sum c_iE_i$, where $c_i$ is a positive integer if $f(E_i)\subseteq S$ and $c_i=0$ otherwise. Writing $K_Y\sim f^*(K_X)+\sum b_iE_i$ we get  $$K_Y+f_*^{-1}(S)\sim f^*(K_X+S)+\sum_{E_i, f(E_i)\not\subseteq S} b_iE_i +\sum_{E_i, f(E_i)\subseteq S} (b_i-c_i)E_i.$$ Since $(X,S)$ is plt in a neighborhood of $S$, $a_{X,S}(E_i)>-1$ for every $E_i$. It follows that for every $E_i$ such that $f(E_i)\subseteq S$ we have $b_i=a_{X,S}(E_i)+c_i>0$. Thus $X$ is terminal in a neighborhood of $S$ if and only if for every $E_i$ such that $f(E_i)\not\subseteq S$ the discrepancy $b_i=a_{X,S}(E_i)$ is positive. 
\end{proof}

Let us recall from \cite[Chapter 6]{Matsumura} basic properties concerning regular sequences and related notions. Let $(R,\fm)$ be a Noetherian local ring. Recall that a sequence $(x_1,\ldots,x_r)$ of elements of $\fm$ is \emph{regular} if for every $i=1,\ldots, r$ the element $x_i$ is not a zero divisor in $R/(x_1,\ldots,x_{i-1})$. The \emph{depth} and the \emph{(Krull) dimension} of $R$ are defined respectively as the maximal length of a regular sequence and as the maximal number of strict inclusions of prime ideals. The ring is called \emph{Cohen-Macaulay} if $\fm$ contains a regular sequence of length $\dim R$. Equivalently, $R$ satisfies Serre's conditions $S_i$: $\depth R\geq \min(\dim R,i)$, for all $i\geq 1$. The ring is \emph{regular} if $\fm$ contains a regular sequence of length $\dim R$ generating $\fm$. We say that a prime ideal $I\subseteq \fm$ of $R$ is a \emph{complete intersection} if it is generated by a regular sequence of length equal to its height. A scheme $X$ is called Cohen-Macaulay if all its local rings are Cohen-Macaulay. An irreducible closed subscheme $Y$ of a scheme $X$ is called a \emph{local complete intersection in $X$} if the sheaf of ideals of $Y$ is a complete intersection in all local rings of $X$. Finally, we say that an irreducible scheme $X$ is a \emph{local complete intersection} if it is locally isomorphic to a local complete intersection in a smooth scheme. 
 
 \smallskip
Let us also recall that a normal variety $X$ has \emph{rational singularities} if for every resolution $f\:\widetilde X\to X$ of the singularities of $X$ we have $R^if_*\cO_{\widetilde X}=0$ for $i\geq 1$.
We will use the following known facts about rational singularities. Note that Kawamata log terminal singularities are rational, see e.g. \cite[Theorem 6.2.12]{Ishii-Intro_to_singularities}. 

\begin{lem}[Rational singularities]\label{lem:rational_sing} 
Let $X$ be a normal variety with rational singularities. Then the following hold: 

\begin{enumerate}[(a)]
\item $X$ is Cohen-Macaulay and hence every local complete intersection $Y$ in $X$ is Cohen-Macaluay. In particular, $Y$ is normal if and only if its singular locus has codimension  at least $2$.  
\item The group $\Cl(X)/\Pic(X)$ is finitely generated. 
\item If $X$ is projective and $\cN$ is a big and nef invertible sheaf on $X$  then $H^i(X,\cN^\vee)=0$ for $0\leq i\leq \dim X -1$.
\end{enumerate}
\end{lem}

\begin{proof}

(a) By \cite[Theorem 6.2.14]{Ishii-Intro_to_singularities}, $X$ is Cohen-Macaulay. Then by \cite[Proposition 5.3.12]{Ishii-Intro_to_singularities}, $Y$ is Cohen-Macaulay, too. In particular, by Serre's criterion \cite[Theorem 23.8]{Matsumura}, $Y$ is normal if and only if it is regular in codimension $1$.

%To prove (b) and (c) we first make a cohomological observation. Let $f\:Y\to X$ be a morphism of schemes, $\cE$ a locally free sheaf of finite rank on $X$ and $\cF$ a quasi-coherent sheaf on $Y$. By the projection formula \cite[Exercise III.8.3]{Hartshorne}, $R^jf_*(\cF\otimes_{\cO_Y} f^*\cE)\cong R^j f_*(\cF)\otimes _{\cO_X}\cE$, so the Leray spectral sequence gives  
%\begin{equation}\label{eq:Leray}
%$H^i(X,R^jf_*(\cF)\otimes_{\cO_X} \cE)\Rightarrow H^{i+j}(Y,\cF\otimes_{\cO_Y} f^*\cE).$
 %\end{equation}
%If $R^jf_*(\cF)=0$ for $j\geq 1$ then the sequence degenerates, which gives natural isomorphisms for all $i\geq 0$ (cf.\ \cite[Exercise III.8.1]{Hartshorne}):
%\begin{equation}\label{eq:pullback_cohomology}
%H^i(X,f_*(\cF)\otimes_{\cO_X} \cE)\xrightarrow{\cong}H^i(Y,\cF\otimes_{\cO_Y}f^*\cE).
% \end{equation}

(b) This is proved by reducing to the analytification and then to finiteness of singular homology of the resolution using the exponential sequence, see \cite[Lemma 1.1]{Kawamata-Crepant-Bl-3d}, cf. also  \cite[Propositions 12.1.4, 12.1.6]{KollarMori-3d_flips}. 

(c) Since $X$ has rational singularities, for every locally free sheaf of finite rank $\cE$ on $X$ the projection formula and the Leray spectral sequence (see, \cite[Exercises III.8.1, 8.3]{Hartshorne}) give natural isomorphisms $H^i(X,\cE)\xrightarrow{\cong}H^i(\widetilde X,f^*\cE)$ for $i\geq 0$.
Put $\widetilde \cN=f^*\cN$. Then $H^i(X,\cN^\vee)\cong H^i(\widetilde X,\widetilde \cN^{\vee})$. Since $\widetilde \cN$ is big, nef and invertible, $H^i(\widetilde X, \widetilde \cN^{\vee})=H^{\dim X-i}(\widetilde X,\cO(K_{\widetilde X})\otimes \widetilde \cN)=0$ for $i<\dim X$ by Serre's duality and the Kawamata-Viehweg vanishing theorem for smooth varieties. 
\end{proof}

\subsection{Inversion of adjunction and a $\Q$-factorial terminalization} 

We recall the following version of adjunction and inversion of adjunction, see \cite[Remark 5.47 and Theorem 5.50]{KollarMori-Bir_geometry}, cf.\ \cite[Chapter 4]{Kollar-Singularities_of_MMP} and \cite[Chapters 16 and 17]{FA}.

\begin{lem}[Inversion of adjunction]\label{lem:different} Let $(X,S)$ be a log pair such that $S$ is a normal prime Weil divisor which is Cartier in codimension 2. Then the adjunction formula $K_S=(K_X+S)|_S$ holds. Furthermore, $(X,S)$ is plt in a neighborhood of $S$ if and only if $S$ is klt. 
\end{lem}

A \emph{$\Q$-factorial terminalization} of a normal quasi-projective variety $X$ is a proper birational morphism $f\:X'\to X$ such that $X'$ is a quasi-projective $\Q$-factorial terminal variety and $K_{X'}$ is $f$-nef. 

\begin{lem}[$\Q$-factorial terminalization] \label{lem:terminalization}
Every normal quasi-projective variety $X$ has a $\Q$-factorial terminalization $f\:X'\to X$. Furthermore, the restriction of $f$ over every $\Q$-factorial terminal open subset of $X$ is an isomorphism. 
\end{lem}

\begin{proof} By \cite[Theorem 1.33]{Kollar-Singularities_of_MMP} there exists a proper birational morphism $g\:Y\to X$ such that $Y$ is quasi-projective and terminal and $K_Y$ is $g$-nef. By \cite[Corollary 1.37]{Kollar-Singularities_of_MMP} there exists a proper birational morphism $h\:X'\to Y$ such that $X'$ is quasi-projective terminal and $\Q$-factorial and $h$ is small, i.e.\ does not contract any divisor. Then $h\circ g$ is a $\Q$-factorial terminalization.

Given an open subset $U\subseteq X$, the restriction $f|_{f^{-1}(U)}$ is a $\Q$-factorial terminalization of $U$, so without loss of generality we may assume that $U=X$. By assumption $K_{X'}$ is $f$-nef, so the divisor $K_{X'}-f^*K_X=\sum_{E\in  \cE(f)} a_X(E)E$ is $f$-nef. Since $X'$ is $\Q$-factorial and the latter divisor is contracted by $f$, the Negativity Lemma \cite[3.39(1)]{KollarMori-Bir_geometry} gives $a_X(E)\leq 0$ for each $E\in \cE(f)$. Since $X$ is terminal, we infer that $\cE(f)=\emptyset$, that is, $\Exc(f)$ has codimension at least $2$. But $X$ is also $\Q$-factorial, so \cite[VI.1, Theorem 1.5]{Kollar-Rational_curves_on_alg_var} implies that $\Exc(f)=\emptyset$. Thus $f$ is an isomorphism. 
\end{proof}

\subsection{Fano varieties and Mori fiber spaces} 

\begin{defn}[Fano variety and its index] \label{def:Fano} A \emph{Fano variety} is a normal projective variety whose anti-canonical divisor is ample (in particular, $\Q$-Cartier). 
\end{defn}

Let $X$ be a klt Fano variety. By the Kawamata-Viehweg vanishing theorem, see \cite[Theorem 2.70]{KollarMori-Bir_geometry} or \cite[Theorem 5.2.7]{Matsuki}, we have  $H^i(X,\cO_X)=0$ for all $i>0$. The linear equivalence on $X$ coincides with numerical and homological equivalence \cite[Proposition 2.1.2]{Iskovskikh_Prokhorov-Fano_varieties}. In particular, $\Pic X\cong H^2(X,\Z)\cong \NS(X)$. It is also known that $X$ is simply connected \cite{Takayama-pi_1(lt_Fano)} and rationally connected \cite[Theorem 1]{Zhang-Fanos_rationally_connected}.

\medskip

Recall \cite[\S 2.1 and Example 19.1.4]{Fulton-Intersection_theory} that for a (possibly non-normal) complete algebraic variety $X$ the quotient $\NS(X)$ of the Picard group of $X$  by the relation of numerical equivalence is a finitely generated free abelian group, whose rank $\rho(X)$ is called the \emph{Picard rank} of $X$. For a surjective morphism of complete varieties $f\:X\to B$ we put $\rho(X/B)=\rho(X)-\rho(B)$. 

Note that a Fano variety of Picard rank one is not necessarily $\Q$-factorial in general, see for instance Example \ref{ex:Q-fact} below. In contrast, a Fano variety $X$ with class group $\Cl(X)\cong \Z$ is automatically $\Q$-factorial, as the image of the natural inclusion $\Pic(X)\to \Cl(X)$ is a nontrivial subgroup of finite index. For a Fano variety $X$ with class group $\Cl(X)\cong \Z$ the \emph{Fano index} of $X$ is the positive integer $i_X$ such that $-K_X\sim i_X H$ for some ample generator $H$ of $\Cl(X)$.

\medskip

A morphism $f\:X\to B$ between quasi-projective varieties is called a \emph{contraction} if it is proper, surjective and $f_*\cO_X=\cO_B$. The latter condition implies connectedness of fibers of $f$ and in case $B$ is normal it is equivalent to it.

\begin{Rem}[Contractions from varieties with $\rho=1$] \label{rem:trivial contractions} 
A projective variety of Picard rank one has only trivial contractions. Indeed, let $f\:X\to B$ be a contraction from such a variety onto a positive dimensional variety $B$.  Note that since $f$ is proper and $B$ is quasi-projective, $B$ is projective. Let $H$ be an effective ample Cartier divisor on $B$. By Kleiman's criterion a divisor numerically equivalent to an ample divisor is ample, so since $\rho(X)=1$, $f^*H$ is an effective ample Cartier divisor on $X$. For every irreducible curve $C$ on $X$, it follows from the projection formula  \cite[Proposition 2.3(c)]{Fulton-Intersection_theory} that $H\cdot f_*(C)=(f^*H)\cdot C>0$. Thus, $f\:X\to B$ is a proper morphism which does not contract any curve, hence is a finite morphism.  Since  $f_*\cO_X=\cO_B$ by assumption, $f$ is an isomorphism.
\end{Rem}

\begin{defn}[Mori fiber spaces and completions]\label{dfn:MFS}\
\begin{enumerate}[(a)]
\item A  \emph{Mori fiber space} is a $\Q$-factorial terminal projective variety $X$ endowed with a contraction $f\:X\to B$ onto a lower-dimensional normal variety $B$ such that $\rho(X/B) = 1$ and $-K_X$ is $f$-ample. 
\item  Two Mori fiber spaces $f_i\: X_i \to B_i$, $i=1,2$,  are called \emph{weakly square birational equivalent} if there exist birational maps $\varphi\: X_1\map X_2$ and $\varphi'\: B_1 \map B_2$ such that $f_2 \circ \varphi = \varphi' \circ f_1$.
\item Given a quasi-projective variety $U$, a \emph{Mori fiber completion of $U$} is a Mori fiber space whose total space is a completion of $U$.
\end{enumerate}
\end{defn}

It follows from the definition that general fibers of a Mori fiber space are Fano varieties. Since the total space is assumed to be terminal,  by \cite[Proposition 7.7]{Kollar-Singularities_of_pairs} general fibers are terminal too. We note that weakly square birational equivalent Mori fiber spaces are \emph{square birational equivalent} in the sense of \cite[Definition 1.2]{Corti-Sing_of_lin_sys} if the induced morphism on generic fibers is an isomorphism.

\smallskip We have the following notion of rigidity of varieties. See \cite{Cheltsov-rigid_Fano} and \cite{Pukhlikov-bir_rigid} for related results.

\begin{defn}[Birationally rigid varieties]\label{def:rigid}
A Fano variety is called \emph{birationally rigid} if it has no birational maps to Mori fiber spaces other than its own birational automorphisms. It is called \emph{super-rigid} if additionally all its birational automorphisms are regular.
\end{defn}

In particular, positive-dimensional birationally rigid varieties are non-rational. We have the following analogous affine notions, see \cite{CDP17}.

\begin{defn}[Affine Fano varieties]\label{def:Affine-Fano}  An \emph{affine Fano variety} is an affine variety which admits a completion by a purely log terminal log pair $(X,S)$ such that $X$ is a (normal projective) $\Q$-factorial variety of Picard rank one, $S$ is prime and $-(K_X+S)$ is ample.
\end{defn}

\begin{defn}[Affine super-rigid Fano varieties]\label{def:affine_rigid}
An affine Fano variety $U$ is \emph{super-rigid} if it satisfies the following conditions:
\begin{enumerate}[(a)]
\item $U$ does not contain Zariski open subsets which are relative affine Fano varieties over varieties of positive dimension.  
\item For every completion $(X,S)$ of $U$ and every log pair $(X',S')$ as in Definition \ref{def:Affine-Fano}, if there exists an isomorphism $U\cong X'\setminus S'$ then it extends to an isomorphism $X\cong X'$ mapping $S$ onto $S'$. 
\end{enumerate}
\end{defn} 

Note that $\A^1$ is the only affine Fano curve and it is super-rigid. It follows from the definition that a super-rigid affine Fano variety of dimension $\geq 2$ does not contain open $\A^1$-cylinders, that is, open subsets isomorphic to the product of $\A^1$ with a variety of smaller dimension.

\medskip \section{Pencils and their resolutions}\label{sec:Weil-Pencils}

We recall the correspondence between dominant rational maps $\psi\:X\map \P^1$ on a normal variety $X$ and linear systems of Weil divisors of projective dimension $1$ on $X$. In the smooth case it restricts to the  well-known correspondence between such maps and one-dimensional linear systems of Cartier divisors, see e.g.\ \cite{Dolgachev-Classical_AG_modern_view}. For this purpose we use the correspondence between Weil divisors and coherent reflexive sheaves of rank one, see \cite[Section 5.2]{Ishii-Intro_to_singularities}, \cite{Reid-canonical_3folds} or \cite[Appendix]{Schwede-reflexive}. For a sheaf $\cF$ on $X$ we denote the sheaf $\HHom_{\cO_X}(\cF,\cO_X)$, dual to $\cF$, by $\cF^{\vee}$.

\subsection{Divisorial sheaves}\label{ssec:Divisorial_sheaves}

Let $X$ be a normal variety and let $\cK_{X}$ be its sheaf of rational functions. For a Weil divisor $D$ on $X$ the \emph{divisorial sheaf associated to} $D$ is the unique subsheaf of $\cO_{X}$-modules of $\cK_{X}$ whose sections over every open subset $U$ of $X$ are \[\cO_{X}(D)(U)=\{f\in\cK_{X}^{*},\,\div(f)|_{U}+D|_{U}\geq0\}\cup\{0\}.\]  The sheaf $\cO_{X}(D)$ is a coherent reflexive sheaf of rank one. It is invertible if and only if $D$ is Cartier. If $D$ is effective then $\cO_X(-D)= \cO_X(D)^\vee$ is the ideal sheaf $\cI_D$ of $D$, which is a coherent reflexive subsheaf of $\cO_X$ of rank one. Conversely, every coherent reflexive sheaf of rank one $\cF$ on $X$ embeds into $\cK_{X}$ and for each embedding $i\:\cF\hookrightarrow\cK_{X}$  there is a unique Weil divisor $D$ on $X$ such that $\Im(i)=\cO_{X}(D)$. We henceforth use the term \emph{divisorial sheaf} to refer to any coherent reflexive sheaf of rank one on $X$. We note that $\cO_{X}(0)=\cO_{X}$, because on a normal variety a rational function with no poles is regular, \cite[Theorem 11.5]{Matsumura}. More generally, given a divisorial sheaf $\cF$ and an open subset $j\:U\hookrightarrow X$ such that $\codim(X\setminus U)\geq2$, the natural homomorphism $\cF\to j_{*}(\cF|_{U})$ is an isomorphism. 

The correspondence $D \mapsto \cO_X(D)$ induces an isomorphism between the class group $\Cl(X)$ of $X$ and the set of isomorphism classes of divisorial sheaves endowed with the group law defined by $\cF\hat{\otimes}\cF'=(\cF\otimes \cF')^{\vee\vee}$. In case $\cF$ or $\cF'$ is invertible, there is a canonical isomorphism $\cF\otimes\cF'\cong (\cF\otimes\cF')^{\vee\vee}$. The inclusion of the smooth locus $j\:X_\reg\hookrightarrow X$ induces an isomorphism $$j^*\:\Cl(X)\to \Cl(X_\reg)\cong \Pic(X_\reg),$$ whose inverse is given by associating to an invertible sheaf $\cN$ on $X_\reg$ the divisorial sheaf $(j_*\cN)^{\vee\vee}$ on $X$. The canonical sheaf of $X$ is $\omega_X=(j_*\det(\Omega^1_{X_\reg}))^{\vee\vee}\cong \cO_X(K_X)$.

\begin{example}[The quadric cone in $\P^3$]\label{ex:quadric_cone} Let $\P$ be the projective cone in $\P^3$ over a smooth plane conic. It is isomorphic to the weighted projective plane $\P(1,1,2)$ with weighted homogeneous coordinates $x_0$, $x_1$, $x_2$; see Section \ref{sec:Cone_construction}. Let $H=\{x_0=0\}$.  Since $H\cong \P^1$ is irreducible and $\P\setminus H\cong \A^2$ has a trivial class group, the class group of $\P$ is isomorphic to $\Z$ and is generated by $H$. The divisor $2H$ is Cartier, but $H$ itself is not Cartier, equivalently, the divisorial sheaf $\cO_{\P}(H)$ is not invertible. Indeed, the open subset $U=\P\setminus \{x_2=0\}$ is isomorphic to the affine quadric cone $\{x^2-yz=0\}\subseteq \A^3$ and the restriction of $\cO_{\P}(H)$ to $U$ is isomorphic to the divisorial sheaf $\cO_U(D)$ associated to the Weil divisor $D=\{x=y=0\}$. The latter is not Cartier, since the ideal $(x,y)=H^0(U, \cO_U(-D))$ of $R=\C[U]$ is non-principal.  

Note that $H^0(U, \cO_U(D))$ is equal to the $R$-submodule of $\mathrm{Frac}(R)$ generated by $1$ and $\frac{z}{x}$ and that, putting $D'=\{z=x=0\}$ and $s=\tfrac{z}{x}$, we have $\div(s)=\div(z)-\div(x)=2D'-(D'+D)=D'-D$. In particular, $s\in H^0(U, \cO_U(D))$ and $D'\sim D$. Note also that the natural $\cO_U$-linear homomorphism $\cO_U(-D)\otimes \cO_U(-D)\to \cO_U(-2D)$ is not an isomorphism, as on global sections it gives the inclusion of ideals $(xy,y^2,x^2)=(y)\cdot (x,y,z)\hookrightarrow (y)$.
\end{example}

\subsection{Pencils of Weil divisors} %\label{ssec:resolutions_of_pencils} 

Let $X$ be a normal variety and let $\psi\: X \map \P^1$ be a dominant rational map. Let $j\:U=\dom(\psi)\hookrightarrow X$ be the inclusion of the domain of definition of $\psi$ and let $\psi_U:=\psi|_U\:U\to \P^1$. Since $X$ is normal, $X\setminus U$ is a closed subset of codimension $\geq 2$ of $X$. Put  $V=H^0(\P^1,\mathcal{O}_{\P^1}(1))$. The invertible sheaf $\cN=\psi_U^*\mathcal{O}_{\P^1}(1)$ on $U$ extends to a unique divisorial sheaf $\cF=(j_*\cN)^{\vee\vee}$ on $X$ and the sections $\psi_U^*(s)$, $s\in V$, extend uniquely to global sections of $\cF$. The obtained homomorphism $\psi^*\:V\to H^0(X,\cF)$ is injective and its image is a $2$-dimensional linear subspace $\caL\subseteq H^0(X,\cF)$. 

The scheme-theoretic fibers $U_p=\psi_{U}^{*}(p)$ of $\psi_{U}$ over the closed points $p$ of $\P^1=\Proj(\Sym^{\boldsymbol{\cdot}}(V))$ are linearly equivalent Cartier divisors on $U$ for which  $\cO_{U}(U_p)\cong\psi_{U}^{*}\cO_{\P^{1}}(1)$. For every $p\in\P^{1}$ we denote by $\caL_p$ the scheme-theoretic closure of $\psi_{U}^{*}(p)$ in $X$. Since $X$ is normal and $\codim_{X}(X\setminus U)\geq2$, these are linearly equivalent Weil divisors on $X$ for which $\cO_{X}(\caL_p)\cong\cF$. Denote by $\cI_{U_p}\subseteq\cO_{U}$ the ideal sheaf of $U_p$. Then the ideal sheaf $\cI_{\caL_p}\subseteq\cO_{X}$ of $\caL_p$ is equal to $(j_{*}\cI_{U_p})^{\vee\vee}\cong\cO_{X}(-\caL_p)\cong\cF^{\vee}$. 
 
In what follows, we call the subspace $\caL\subseteq H^0(X,\cF)$ the \emph{pencil of (Weil) divisors on $X$} defining the rational map $\psi\:X\dashrightarrow \P^1$. The divisors $\caL_p$, $p\in \P^1$ are called the \emph{members} of $\caL$.  The \emph{base scheme} of the pencil, denoted by $\Bs \caL$, is the scheme-theoretic intersection in $X$ of all members of the pencil. 
 
\begin{prop}[Pencils and their associated rational maps]\label{lem:pencil-as-vectorspaces}
Let $X$ be a normal algebraic variety. There exists a natural bijection between the set of dominant rational maps $\psi\:X\dashrightarrow \P^1$ and the set of equivalence classes of pairs $(\cF,\caL)$, where $\cF$ is a divisorial sheaf on $X$ and $\caL\subseteq H^0(X,\cF)$ is a $2$-dimensional space of global sections generating $\cF$ off a closed subset of codimension $\geq 2$, where two such pairs  $(\cF,\caL)$ and $(\cF',\caL')$ are equivalent if there exists an isomorphism $\alpha\:\cF \to \cF'$ for which $H^0(\alpha)(\caL)=\caL'$. 
 \end{prop} 

\begin{proof}
We already described above how to associate to a dominant rational map $\psi\:X\dashrightarrow \P^1$ a pair $(\cF,\caL)=(\cF_\psi,\caL_\psi):=((j_*\cN)^{\vee\vee},\psi^*V)$. Consider the natural homomorphism 
\begin{equation}\label{eq:e}
\mathrm{e}\:\caL\otimes_\C \cO_X\to \cF
\end{equation}
defined by restricting global sections to stalks of $\cF$. By construction, the restriction of $\psi$ to $U=\dom(\psi)$ is isomorphic to the pull-back by $\psi_U\:U\to \P^1$ of the canonical surjection $V\otimes_\C \cO_{\P^1}\to \cO_{\P^1}(1)$. Thus the sections contained in $\caL$ generate $\cF$ outside the indeterminacy locus $X\setminus U$ of $\psi$, which is a closed subset of codimension $\geq 2$ of $X$.

Conversely, let  $\cF$ be a divisorial sheaf on $X$ and let $\caL\subseteq H^0(X,\cF)$ be  a $2$-dimensional space of global sections such that the support $Z$ of the cokernel of the homomorphism \eqref{eq:e} has codimension $\geq 2$ in $X$. Since $\cF$ is divisorial  and $X$ is normal, the set $\cF_{\mathrm{sing}}$ of points $x$ of $X$ such that $\cF_{x}$ is not a free $\cO_{X,x}$-module  is a closed subset of codimension $\geq2$. Thus, $W:=W(\cF,\caL)=X\setminus(Z\cup\cF_{\mathrm{sing}})$  is an open subset of $X$ with a complement of codimension $\geq2$, on which $\mathrm{e}$ restricts to a surjection $\mathrm{e}|_W\:\cO_{W}^{\oplus2}\cong \caL\otimes_\C\cO_{W}\to\cF|_{W}$ onto the invertible sheaf $\cF|_{W}$. By \cite[Proposition 7.12]{Hartshorne} there exists a unique dominant morphism $f=f_{\cF,\caL}\:W\to \P^1$ such that $\mathrm{e}|_W$ is equal to the  pull-back by $f$ of the canonical surjection $V\otimes_\C \cO_{\P^1}\to \cO_{\P^1}(1)$. This morphism determines in turn a unique rational map $\psi\:X\map \P^1$ whose domain of definition $U$ contains $W$ and for which $\psi_U|_W=f$. 

Two pairs $(\cF,\caL)$ and $(\cF',\caL')$ determine the same dominant rational map $\psi\:X\dashrightarrow \P^1$ if and only if their associated morphisms $f_{\cF,\caL}\:W(\cF,\caL)\to \P^1$ and $f_{\cF',\caL'}\:W(\cF',\caL')\to \P^1$ coincide on the open subset $\widetilde{W}=W(\cF,\caL)\cap W(\cF',\caL')$. This is in turn equivalent to the existence of an isomorphism $\alpha_{\widetilde{W}}\:\cF|_{\widetilde{W}}\to \cF'|_{\widetilde{W}}$ of sheaves on $\widetilde{W}$ which maps the global sections $s|_{\widetilde{W}}$ of $\cF|_{\widetilde{W}}$, $s\in \caL$, bijectively onto the global sections $s'|_{\widetilde{W}}$ of $\cF'_{\widetilde{W}}$, $s'\in \caL'$. Since $X$ is normal, $\codim_X(X\setminus \widetilde{W})\geq 2$ and $\cF$ and $\cF'$ are reflexive, $\alpha_{\widetilde{W}}$ uniquely extends to an isomorphism $\alpha\:\cF \to \cF'$ of sheaves over $X$ such that $H^0(\alpha)(\caL)=\caL'$. So the association $(\cF,\caL)\mapsto \psi_{\cF,\caL} $ induces a well-defined injective map from the set of equivalence classes of pairs $(\cF,\caL)$ to the set of dominant rational maps $\psi\:X\map \P^1$. This map is also surjective, because the equality $\psi=\psi_{\cF_\psi,\caL_\psi}$  holds for every dominant rational map $\psi\:X\map \P^1$. 
\end{proof}

By a \emph{resolution of $\caL$} we mean a resolution of the associated rational map $\psi$. Given two  linearly equivalent Weil divisors $D$ and $D'$ on $X$ without common irreducible component, the \emph{pencil $\langle D,D'\rangle$ generated by $D$ and $D'$} is the pencil of divisors on $X$, unique up to an isomorphism, which has $D$ and $D'$ among its members. Its  base scheme is equal to the scheme-theoretic intersection of $D$ and $D'$.

\subsection{The graph of a pencil}% \label{ssec:resolutions_of_pencils} 

Let $\psi\:X\map \P^1$ be a dominant rational map on a normal variety $X$. The \emph{graph of $\psi$} is the scheme-theoretic closure $\Gamma\subseteq X \times \P^1$ of the graph of the restriction of $\psi$ to its domain of definition. We let $\gamma\:\Gamma\to X$ and $\pp\:\Gamma\to \P^1$ be the restrictions of the projections from $X\times \P^1$ onto its factors. We obtain a commutative diagram 
\begin{equation}\label{fig:graph}
\xymatrix
   { & \Gamma \ar[dl]_{\gamma} \ar[dr]^{\pp} & \\
     X \ar@{-->}[rr]_{\psi} & & \P^1
   }
\end{equation}
The proper birational morphism  $\gamma\:\Gamma \to X$ provides a natural resolution of $\psi$ such that $\psi\circ\gamma=\pp$. The next proposition collects properties of this resolution. 

\begin{prop}[Properties of the graph resolution]\label{lem:canonical-resol} \label{lem:graph-reso-members} 
Let $X$ be a normal variety. Let $\psi\:X\map \P^1$ be a dominant rational map and let $\caL$ be the associated pencil of divisors. 
Then the following hold: 
\begin{enumerate}[(a)]
\item For every resolution $\tau\:X'\to X$ of $\psi$ the birational map $\gamma^{-1}\circ\tau\:X'\map \Gamma$ is a proper morphism. 
\item The indeterminacy locus of $\psi$ is equal to the support of the base scheme $\Bs\caL$. 
\item The morphism $\gamma$ restricts to an isomorphism over $X\setminus \Bs\caL$, and $\gamma^{-1}(x)\cong \P^1$ for every $x\in \Bs\caL$.
\item For every point $p\in \P^1$ the birational morphism $\gamma_p\:\pp^*(p) \to \caL_p$ induced by $\gamma$ is an isomorphism. 
\end{enumerate}
\end{prop}

\begin{proof}
Set $U=\dom(\psi)$ and let $\pr_X\:X\times \P^1\to X$ denote the projection. 

(a) Since $\psi\circ \tau$ is a morphism, we have a morphism $\tau\times (\psi\circ \tau)\: X'\to X\times \P^1$. Since $X'$ is irreducible and $\tau$ is surjective, the image is contained in $\Gamma$, so we may write this morphism as a composition of some morphism $\tau'\:X'\to \Gamma$ with the closed immersion $\Gamma\hookrightarrow X\times \P^1$. Both morphisms $\gamma$ and $\tau=\gamma\circ\tau'$ are proper, so $\tau'$ is proper too.

(b) With the notation of the proof of Proposition \ref{lem:pencil-as-vectorspaces}, $\psi$ is represented by the morphism $f\:W\to \P^1$ associated to the restriction of $\mathrm{e}\:\caL\otimes_\C \cO_X\to \cF$ to the open subset $W=X\setminus (\Supp(\Coker\ e)\cup \cF_{\Sing})$. Let $x$ be a point of $X$ which is not contained in the support of some member $\caL_p$. Then $\cI_{\caL_{p},x}=\cO_{X,x}$, and since  $\cI_{\caL_{p}}\cong \cF^{\vee}$, it follows that $\cF^{\vee}_x$, and hence $\cF_x$ is a free $\cO_{X,x}$-module. Thus $\cF_{\Sing} \subseteq \Supp(\Coker\ e)$ and we have $U=X\setminus \Supp(\Coker\ e)=X\setminus \Supp (\Bs\caL)$.

(c) By the definition of $\Gamma$, $\gamma$ restricts to an isomorphism over $U$. Suppose that for some $x\in \Bs\caL$, $\pr_X^{-1}(x)\cong \P^1$ is not fully contained in $\Gamma$. Since $X$ is normal and $\gamma\:\Gamma\to X$ is proper and birational, it follows from \cite[Proposition 4.4.1]{EGAIII-1} that the fibers of $\gamma$ are connected, and hence $\gamma^{-1}(x)=\pr_X^{-1}(x)\cap \Gamma$ consists of a unique point $y$, and that there exists an open neighborhood $V$ of $y$ such that  $\gamma|_V\:V\to X$ is an open immersion.  Thus the birational map $\gamma^{-1}$ is defined at $x$, and so is $\psi=\pp\circ\gamma^{-1}$. But this is impossible, because $x\in \Bs\caL=X\setminus U$.

(d) The assertion is local over $X$, so we can assume without loss of generality that $X$ is affine, say  $X=\Spec(A)$. Then $\psi$ is induced by some rational function $h\in \Frac(A)$. Given a representative $h=f/g$, where $f\in A$, $g\in A\setminus\{0\}$, let $U_{(f,g)}=X\setminus V(f,g)$. The restriction of $\psi$ to $U_{(f,g)}$ is given by $x\mapsto [f(x):g(x)]$. For every point $p=[\lambda:\mu]\in \P^1$ the restriction of the Cartier divisor $\psi_U^*(p)$ to $U_{(f,g)}$ is the zero scheme of the regular function $s_{(f,g)}=\mu f-\lambda g$. Letting $\cS=\{(f,g)\in A\times (A\setminus\{0\}),\; h=f/g\}$ we have $U=\bigcup_{(f,g)\in \cS}U_{(f,g)}$. The scheme-theoretic closure $\caL_p\subseteq X$  of $\psi_U^*(p)$ is defined by the vanishing of all functions $s_{(f,g)}$. On the other hand, the graph $\Gamma\subseteq X\times \P^1_{[u:v]}$ is defined by the vanishing of all sections $$\widetilde{s}_{(f,g)}=fv-gu\in H^0(X\times \P^1,\cO_{X\times \P^1}(1)),$$ and hence $\pp^*(p)\subseteq X\times \P^1$ is defined by the vanishing of the section $\mu u- \lambda v$ and all $\widetilde{s}_{(f,g)}$. From this, we see directly that  $\gamma_p\:\pp^*(p)\to X_p$ is an isomorphism. 
\end{proof}

\begin{Rem} Let $X$ be a normal variety and let $\caL\subseteq H^0(X,\cF)$ be a pencil. The defining ideal sheaf $\cI_\caL$ of the base scheme $\Bs\caL$ can be described as follows. Consider the homomorphism 
\begin{equation}\label{eq:ev}
\ev\:\caL\otimes_\C \cF^\vee\to \cO_X
\end{equation}
obtained from the evaluation homomorphism $\mathrm{e}\:\caL\otimes_\C \cO_X\to \cF$  \eqref{eq:e} by tensoring it with $\cF^{\vee}$ and composing with the canonical homomorphism $\cF\otimes_{\cO_X} \cF^{\vee} \to \cO_X$. Then the ideal sheaf $\cJ=\Im(\ev)$ is generated by the ideal sheaves $\cI_s$ of the zero schemes of the sections $s$ of $\cF$, $s\in \caL$. On the other hand, the members of $\caL$ are, by definition, the Weil divisors associated to the zero schemes of these sections, that is, the closed subschemes of $X$ with defining ideal sheaves $\cI_s^{\vee\vee}$, $s\in \caL$. So $\cI_\caL$ is generated by the ideal sheaves $\cI_s^{\vee\vee}$, $s\in \caL$. It follows that $\cJ\subseteq \cI_\caL$, with equality in case when  $\cF$ is invertible. Indeed, if $\cF$ is invertible then each $\cI_s$ is an invertible ideal sheaf, so $\cI_s=\cI_s^{\vee\vee}$.
\end{Rem}

For a pencil $\caL$ on a normal variety $X$, Proposition \ref{lem:canonical-resol}(b) says that $\gamma\:\Gamma \to X$ is a universal minimal resolution of the dominant rational map $\psi\:X\map \P^1$ determined by $\caL$, in the sense that every resolution of $\psi$ factors through it. Another natural resolution of $\psi$ is given by the blow-up $\tau\:\Bl_{\Bs\caL}(X)\to X$ of the base scheme of $\caL$, as shown in the following lemma. In case when members of $\caL$ are Cartier, it is a classical fact that the induced birational morphism  $\gamma^{-1}\circ \tau\:\Bl_{\Bs\caL}(X)\to\Gamma$ is an isomorphism; see e.g.\  \cite[Proposition 7.12 and \S 7.1.3]{Dolgachev-Classical_AG_modern_view}. This is no longer true in general for pencils whose members are not Cartier, see Example \ref{ex:non-iso-Blowup-discr}.

\begin{lem}[The blowup of $\Bs \caL$ is a resolution] Let $X$ be a normal variety. Let $\psi\:X\dashrightarrow \P^1$ be a dominant rational map determined by a pencil of divisors $\caL$. Then the blow-up  $\tau\:\Bl_{\Bs\caL}(X)\to X$ of the base scheme of $\caL$ is a resolution of $\psi$.
\end{lem}

\begin{proof} Let $\widetilde{X}=\Bl_{\Bs\caL}(X)$. To verify that $\psi\circ \tau\:\widetilde{X}\map \P^1$ is a morphism we can assume without loss of generality that $X=\Spec(A)$ is affine and  that $\psi$ is the rational map defined by some rational function $h\in\Frac(A)$. With the notation of the proof of Proposition \ref{lem:canonical-resol}(c), $\Bs\caL$ is the closed subscheme of $X$ whose defining ideal $I$ generated by the regular functions $f$ and $g$ such that $(f,g)\in \cS$. By definition of the blow-up, the ideal sheaf $\cJ=\tau^{-1}(I)\cdot \cO_{\widetilde{X}}$ is invertible. It is generated by all regular functions $\tau^*f$ and $\tau^*g$, where $(f,g)\in \cS$. Thus for every point  $y_0\in \widetilde{X}$ there exists an element $(f,g)\in \cS$ such that the stalk $\cJ_{y_0}$ is generated by $\tau^*f$ or $\tau^*g$, say $\tau^*f$, the situation being symmetric for $\tau^*g$. It follows that there exists an open neighborhood $V\subseteq \widetilde{X}$ of $y_0$ and a regular function $\widetilde{g}$ on $V$ such that $\tau^*g|_V=\widetilde{g}\tau^*f|_V$. On this neighborhood the composition $\psi\circ \tau\:\widetilde{X}\map \P^1$ is given by $$y\mapsto [(\tau^*f)(y):(\tau^*g)(y)]=[1:\widetilde{g}(y)],$$ so $y_0\in \dom(\psi\circ \tau)$. Thus $\psi\circ \tau\:\widetilde{X}\map \P^1$ is a morphism, and hence $\tau\:\widetilde{X}\to X$ is a resolution of $\psi\:X\map \P^1$. 
\end{proof}

\smallskip \section{Terminal rank one Fano pencils and associated Mori fiber spaces} \label{sec:pencils_and_resolutions}

Our strategy to construct Mori fiber completions of affine spaces is to use pencils of Weil divisors on normal projective completions of $\A^n$'s and to produce the desired Mori fiber spaces as outputs of relative MMP's ran from suitable resolutions of these pencils. For this approach to work we need in particular to find properties of a variety $X$ and of the members of the pencil whose combination guarantees that a specific open subset is preserved under appropriate choices of resolutions and  relative MMP's.

\subsection{$\Q$-factorial terminal resolutions}\label{sub:QFactTerm-res}\label{subsec:terminal-pencil}

Let $X$ be a normal variety. Let $\caL$ be a pencil of divisors determining a dominant rational map $\psi\:X\map \P^1$. Let $\tau\:X'\to X$ be a resolution of $\caL$, that is, a proper birational morphism from a variety such that $\psi\circ \tau$ is a morphism. By Proposition \ref{lem:canonical-resol} there exists a unique morphism $\Gamma(\tau)\:X'\to \Gamma$ such that $\tau=\gamma\circ\Gamma(\tau)$ and a commutative diagram
\begin{equation}\label{fig:graph}
\xymatrix
   {X' \ar@{->}[rr]^{\Gamma(\tau)} \ar[dr]_{\tau}& & \Gamma \ar[dl]_{\gamma} \ar[dr]^{\pp} & \\
    & X \ar@{-->}[rr]_{\psi} & & \P^1
   }
\end{equation}
where $(\Gamma,\gamma,\pp)$ are as in  \eqref{fig:graph}. Let $\nu\:\widetilde\Gamma\to \Gamma$ be the normalization of $\Gamma$. We call $\widetilde\Gamma$ the \emph{normalized graph} of $\psi$. If $X'$ is normal then we have $\Gamma(\tau)=\nu\circ \widetilde\Gamma(\tau)$ for some unique birational proper morphism $$\widetilde\Gamma(\tau)\:X'\to \widetilde\Gamma.$$

\begin{defn}[A thrifty resolution] \label{def:thrifty-reso} Let $X$ be a normal variety and let $\tau\:X'\to X$ be a resolution of a pencil $\caL$ on $X$.
\begin{enumerate}[(a)]
\item We call the image $\delta(\tau)\subseteq\P^1$ of $\Exc \Gamma(\tau)$ by $\pp\circ\ \Gamma(\tau)$ the \emph{discrepancy locus of $\tau$}. 
\item We say that $\tau$ is \emph{$\Q$-factorial terminal} if $X'$ is $\Q$-factorial terminal. 
\item We say that a $\Q$-factorial terminal resolution $\tau$ is \emph{thrifty} if $\widetilde \Gamma(\tau)\:X'\to \widetilde{\Gamma}$ is a $\Q$-factorial terminalization.
\end{enumerate}
\end{defn}

The discrepancy locus is a rough measure of how much a given resolution of $\psi\:X\dashrightarrow \P^1$ differs from the (minimal) graph resolution.

\begin{example}[The affine cone in $\A^4$]\label{ex:non-iso-Blowup-discr} On the affine cone $X=\{xv-yu=0\}\subseteq \A^4$ the Weil divisors $\caL_0=\{u=v=0\}$ and $\caL_\infty=\{x=y=0\}=\caL_0+\div(x/u)$ generate a pencil $\caL$, whose base scheme $\Bs \caL$ is equal to the isolated singular point $p=(0,0,0,0)\in X$. Since the latter has codimension $3$ in $X$, $\caL_0$ is not $\Q$-Cartier. In particular, $X$ is not $\Q$-factorial. The associated rational map is \[\psi_\caL\:X\map  \P^1_{[w_0:w_1]},\quad (x,y,u,v)\mapsto [u:x]=[v:y].\] Its graph $\Gamma$ is isomorphic to the sub-variety of $X\times \P^1_{[w_0:w_1]}$ defined by the equations \[ xw_0-uw_1=0 \quad \textrm{and} \quad yw_0-vw_1=0.\]
The morphism $\pp\:\Gamma\to \P^1$ is a locally trivial $\A^2$-bundle, so $\Gamma$ is smooth. The morphism $\gamma\:\Gamma\to X$ is a thrifty $\Q$-factorial terminal resolution of $\caL$ with an empty discrepancy locus. It is a small resolution of the singularity $p\in X$ with the exceptional locus consisting of a single curve $\gamma_\caL^{-1}(p)\cong \P^1_{[w_0:w_1]}$. It can be also described as the blow-up of the ideal sheaf of $\caL_0$.

On the other hand, the blow-up $\tau_\caL\:\widetilde{X} \to X$ of $\Bs\caL=\{p\}$ is a resolution of the singularity of $X$ with exceptional divisor $\tau_\caL^{-1}(p)\cong \P^1\times \P^1$. In particular, it is a $\Q$-factorial terminal resolution of $\caL$. The birational proper morphism $\tau':=\Gamma(\tau_\caL)\:\widetilde X\to \Gamma$ contracts $\tau_\caL^{-1}(p)$ onto $\gamma^{-1}(p)$. Since the proper transform by $\tau'$ of every closed fiber of $\pp\:\Gamma\to \P^1$ is isomorphic to the blow-up of the origin in $\A^2$ the discrepancy locus of $\tau_\caL$ is equal to $\P^1$.
 \end{example}

A $\Q$-factorial terminal resolution $\tau\:X'\to X$ with a finite discrepancy locus induces isomorphisms between general fibers of $\pp\:\Gamma\to \P^1$ and their proper transforms on $X'$. The latter are general fibers of  $\pp\circ\ \Gamma(\tau)\:X'\to \P^1$, and since $X'$ is terminal, they are terminal varieties by \cite[Proposition 7.7]{Kollar-Singularities_of_pairs}. On the other hand, by Proposition \ref{lem:graph-reso-members}(d) the fibers of $\pp\:\Gamma\to \P^1$ are isomorphic to the members of the pencil $\caL$ determining $\psi\:X\map \P^1$. The terminality of general members of $\caL$ is thus a necessary condition for the existence of a $\Q$-factorial terminal resolution of $\caL$.  This motivates the following definition: 

\begin{defn}[A terminal $\Q$-factorial pencil]\label{dfn:Q-factorial-terminal-pencil}  A \emph{terminal pencil} (a \emph{$\Q$-factorial terminal pencil}) on a normal variety is a pencil whose general members are terminal (respectively, $\Q$-factorial terminal).  
\end{defn}

In contrast to terminality, the $\Q$-factoriality of general members of $\caL$ is not necessary for the existence of a $\Q$-factorial terminal resolution of $\psi\:X\dashrightarrow \P^1$, as illustrated by the following example. 
 
\begin{example}[$\Q$-factoriality: general vs generic]  \label{ex:Q-fact}
Let $X=\P^4_{[x:y:z:t:w]}$ and let $F$ and $H$ be the projective cone over the smooth conic  $\{xy-z^2=0\}\subseteq \P^2_{[x:y:z]}$ and the hyperplane $\{t=0\}$, respectively. Denote by $\caL$ be the pencil generated by $F$ and $2H$. A general member of $\caL$ is isomorphic to the projective cone $Z$ in $\P^4$ over the quadric surface $\{xy-z^2+t^2=0\}\subseteq \P^3$. The blowup of the vertex is smooth and the exceptional divisor has discrepancy $1$, so $Z$ is a terminal Fano variety. But it is not $\Q$-factorial. Indeed, its class group is isomorphic to $\Z\oplus \Z$ and is generated for instance by the non-$\Q$-Cartier Weil divisor $\{x=t-z=0\}$ and the hyperplane section $Z\cap \{w=0\}$. 

On the other hand, the graph of the rational map $$\psi\:X \map \P^1, \; [x:y:z:t:w]\mapsto [xy-z^2:t^2]$$ determined by $\caL$ is isomorphic to the subvariety $\Gamma\subseteq X\times \P^1_{[u_0:u_1]}$ with equation $(xy-z^2)u_1+t^2u_0=0$. The generic fiber of $\pp\:\Gamma\to \P^1$ is isomorphic to the projective cone $Y\subseteq \P^4_{\C(\lambda)}$ over the quadric threefold in $\P^3_{\C(\lambda)}$ with equation $xy-z^2+\lambda t^2=0$, where $\lambda=u_0/u_1$. As a variety over $\C(\lambda)$, $Y$ is factorial, with the class group $\Cl(Y)=\Pic(Y)\cong \Z$ generated by the hyperplane section $\{w=0\}$. It follows that $\Gamma$ is factorial (in particular $\Q$-factorial), with the class group generated by the proper transform of $H$ and the hyperplane section $\{w=0\}$.   
\end{example}

By the following criterion, on projective varieties the terminality of a pencil is equivalent to the existence of one terminal member. 

\begin{lem}[One terminal member is sufficient] \label{lem:Q-term-crit} A pencil on a normal projective variety which has at least one terminal ($\Q$-factorial terminal) member is a terminal (respectively $\Q$-factorial terminal) pencil. 
\end{lem}

\begin{proof}
Since $X$ is projective and $\P^1$ is a smooth curve, $\pp\:\Gamma\to \P^1$ is a flat projective morphism.  By Proposition \ref{lem:graph-reso-members}(d), $\gamma\:\Gamma \to X$ induces isomorphisms between scheme-theoretic fibers of $\pp$ and members of $\caL$. Assume that for some $p\in \P^1$, $\caL_p$ is terminal. Then $\pp^*(p)$ is terminal and so \cite[Theorem 9.1.14]{Ishii-Intro_to_singularities} implies that general fibers of $\pp$, and hence general members of $\caL$, are terminal. As a consequence, a general fiber of $\pp$ has rational singularities and its singular locus has codimension at least three. By \cite[Theorem 12.1.10]{KollarMori-3d_flips} the $\Q$-factoriality of fibers of $\pp$ is then an open condition on the set of closed points of $\P^1$. So if $\caL_p$ is in addition $\Q$-factorial then general fibers of $\pp$, and hence the general members of $\caL$, are $\Q$-factorial terminal. 
\end{proof}

We now relate properties of members of a pencil in a neighborhood of the base locus to global properties of the graph of its associated rational map in a neighborhood of the exceptional locus of the graph resolution.

 \begin{prop}[Singularities of the graph] \label{lem:neighborhood-control} 
Let $\caL$ be a terminal pencil on a normal variety $X$ and let $Y$ be a member of $\caL$. Put $Y'=\gamma_*^{-1}Y$. Then the following hold:
\begin{enumerate}[(a)]
\item If $Y$ is normal then $\Gamma$ is normal in an open neighborhood of $Y'$.
\item If $Y$ is klt and $K_{\Gamma}$ is $\Q$-Cartier in an open neighborhood of $Y'$ then $\Gamma$ is terminal in an open neighborhood of $Y'$.
\item If $Y$ is klt, smooth in codimension $2$ and $\Q$-factorial then $\Gamma$ is $\Q$-factorial terminal in an open neighborhood of $Y'$. 
\end{enumerate}
\end{prop}

\begin{proof}
By Proposition \ref{lem:graph-reso-members}(d), $\gamma$ restricts to an isomorphism over each member of $\caL$. In particular, since $\caL$ is a terminal pencil, general fibers of $\pp$ are terminal. In all three cases $Y$ is normal and $Y'=\pp^*(\pp(Y'))$ is a prime Cartier divisor on $\Gamma$. By \cite[Corollaire 5.12.7]{EGAIV-2} there exists a normal open neighborhood $V\subseteq \Gamma$ containing $Y'$. This proves (a). 

(b) By Lemma \ref{lem:different} the log pair $(V, Y')$ is plt in a neighborhood of $Y'$. Let $\pi\:V'\to V$ be a log resolution of this pair and let $G$ be any exceptional prime divisor of $\pi$ whose image is not contained in a fiber of $\pp|_V$. Since $\pi$ induces a log resolution of general fibers of $\pp|_V$ and the latter are terminal, it follows that $G$ has positive discrepancy; see e.g.  \cite[Proposition 7.7]{Kollar-Singularities_of_pairs}. This implies by Lemma \ref{lem:plt-to-terminal} that $V$, and hence $\Gamma$, is terminal in an open neighborhood of $Y'$.

(c) Since klt singularities are Cohen-Macaulay by Lemma \ref{lem:rational_sing}(a), $Y'$ satisfies Serre's condition $S_3$. Since $Y'$ is Cartier, arguing as in the proofs of \cite[Corollary 12.1.9, Lemma 12.1.8]{KollarMori-3d_flips}, we conclude that for every Weil divisor $D$ on $V$ there exists an open neighborhood $V(D)\subseteq V$ of  $Y'$ such that $D|_{V(D)}$ is $\Q$-Cartier. As in (b) we get a terminal open neighborhood $W\subseteq V(K_V)$ of $Y'$. Since terminal singularities are rational, the group $\Cl(W)/\Pic(W)$ is finitely generated by Lemma \ref{lem:rational_sing}(b). The intersection of $W$ and the open neighborhoods $V(D_i)$, where the $D_i$ range through a finite set of Weil divisors whose classes generate $\Cl(W)/\Pic(W)$, is then a $\Q$-factorial terminal open neighborhood of $Y'$.
\end{proof}

\begin{nota}\label{nota:open-subsets} For a terminal pencil $\caL$ on a normal variety $X$ and a finite subset $\delta\subseteq \P^1$  we put $\Gamma_\delta=\pp^{-1}(\P^1\setminus \delta)\subseteq \Gamma $ and $ X_{\delta}= X \setminus \bigcup_{p\in \delta} \caL_p \subseteq X$. We define the following property:
\begin{enumerate}
\item[($\mathbf{TQ}_\delta$)] For every $p\in \P^1\setminus \delta$ the member $\caL_p$ is a prime divisor and on some open neighborhood of $\Bs \caL$ in $\caL_p$ it is klt, smooth in codimension $2$ and $\Q$-factorial.
\end{enumerate}
\end{nota}

Note that for a pencil $\caL$ whose general members are $\Q$-factorial terminal the condition ($\mathbf{TQ}_\delta$) holds for the finite set $\delta \subset \P^1$ consisting of points $p$ for which $\caL_p$ is not $\Q$-factorial terminal.

\begin{cor}[Controlling the discrepancy locus]\label{prop:controled-discrepancy}
Let $\caL$ be a terminal pencil on a normal variety $X$ and let $\delta \subset \P^1$ be a finite set.
If  $X_\delta$ is $\Q$-factorial terminal and  $(\mathbf{TQ}_\delta)$ holds then $\Gamma_\delta$ is $\Q$-factorial terminal. Consequently, the discrepancy locus of every thrifty $\Q$-factorial terminal resolution of $\caL$ is contained in $\delta$. 
\end{cor}

\begin{proof}
Let $E=\gamma^{-1}(\Bs \caL)_\redd$ be the exceptional locus of $\gamma$. By assumption $\Gamma_\delta \setminus E \cong X_\delta\setminus \Bs\caL$ is $\Q$-factorial terminal. On the other hand, it follows from Proposition \ref{lem:neighborhood-control} that for every $p\in \P^1\setminus \delta$ the open set $\Gamma_\delta$ is $\Q$-factorial and terminal in a neighborhood of the intersection of $E$ with the proper transform of $\caL_p$. Since the union of such neighborhoods is an open neighborhood of $E\cap \Gamma_\delta$ in $\Gamma_\delta$, it follows that $\Gamma_\delta$ is $\Q$-factorial terminal. The second assertion follows from  Lemma \ref{lem:terminalization}.
\end{proof}

\subsection{Terminal rank one Fano pencils and relative MMPs}\label{sub:FanoRank1}

In this subsection we consider pencils of Weil divisors whose general members are terminal Fano varieties of Picard rank one and the outputs of relative MMP's ran from their resolutions. We keep the notation of subsection \ref{subsec:terminal-pencil}.

\begin{defn}[A terminal rank one Fano pencil]\label{rofp} Let $X$ be a normal projective variety of dimension at least $2$. A \emph{terminal rank one Fano pencil} on $X$ is a pencil $\caL$ whose general members are terminal Fano varieties of Picard rank one. The \emph{degeneracy locus} of $\caL$ is the finite set $\delta(\caL)\subset \P^1$ consisting of points $p$ such that the member $\caL_p$ is either reducible or has Picard rank strictly higher than one.
\end{defn}

It is known that general fibers of Mori fiber spaces can have Picard rank higher than one (see e.g.\ Example \ref{ex:dP_fibr}). But the additional assumption that the Picard rank of general fibers is one, which we impose in Definition \ref{rofp}, allows to control the effect of running relative MMP's on resolutions of such pencils more easily. Even with this restriction there is still a large natural geometric supply of pencils that can be used to construct Mori fiber completions of $\A^n$'s, see Section \ref{sec:Affine-Spaces}.

\begin{defn}[A compatible thrifty resolution]\label{def:Fano-compatible_res}
Let $\caL$ be a terminal rank one Fano pencil on a normal projective variety. A \emph{compatible thrifty resolution} of $\caL$ is a thrifty $\Q$-factorial terminal resolution of $\caL$ (see Definition \ref{def:thrifty-reso}) whose discrepancy locus is contained in the degeneracy locus of $\caL$.
\end{defn}

\smallskip
\begin{example}[Simple low-dimensional examples]\label{ex:low-dim-QT-FanoRk1}\
\begin{enumerate}[(a)]
\item A terminal rank one Fano pencil on $\P^2$ consists of lines and or conics, and the usual minimal resolution of base points of the pencil is a compatible thrifty resolution. 
\item Since $\P^2$ is the only terminal del Pezzo surface of Picard rank one, the only terminal rank one Fano pencils on $\P^3$ are the pencils of planes. Clearly, the base scheme of every such pencil is a line and its blowup gives a compatible thrifty resolution.
\end{enumerate}
\end{example}

Let $X_0$ be a $\Q$-factorial terminal projective variety and let $f_0\:X_0\to \P^1$ be a surjective morphism. Recall \cite[3.31, Example 2.16]{KollarMori-Bir_geometry} that a $K_{X_0}$-MMP $\varphi\:X_0\map X_{m}=\hat{X}$ relative to $f_0$ consists of a finite sequence $\varphi=\varphi_{m}\circ\cdots\circ\varphi_1$ of birational maps \[\xymatrix@R-0.2em@C-1em{ X_{k-1} \ar@{-->}[rr]^-{\varphi_k} & & X_k \\ & \P^1 \ar@{<-}[ul]^*{ f_{k-1}} \ar@{<-}[ur]_*{ f_k} & }\] between $\Q$-factorial terminal projective varieties, where each $\varphi_k$ is associated to an extremal ray of the closure $\NE( X_{k-1}/\P^1)$ of the relative cone of $1$-cycles of $X_{k-1}$ over $\P^1$. The morphisms $f_k\:X_k\to \P^1$ are the induced surjections. Each $\varphi_k$ is either a relative divisorial contraction or a flip whose flipping and flipped curves are contained in fibers of $ f_{k-1}$ and $f_k$. We say that a relative MMP terminates if either $K_{\hat{X}/\P^1}$ is $f_m$-nef or if $f_m$ factors through a Mori fiber space given by a contraction of a extremal ray in $\NE( X_{m}/\P^1)$.

\begin{prop}[Mori Fiber completions from pencils with compatible resolutions]\label{thm:MainThm-Completion}
Let $\caL$ be a terminal rank one Fano pencil on a normal projective variety $X$ and let $\psi_\caL\:X\map \P^1$ be the associated dominant rational map. Assume that $\caL$ has a compatible thrifty resolution $\tau\:X'\to X$. Then there exists a $K_{X'}$-MMP $\varphi\:X'\map\hat{X}$ relative to $ \psi_\caL\circ\tau\:X'\to \P^1$  which terminates. Furthermore, every terminating $K_{X'}$-MMP relative to $\psi_\caL\circ\tau$ restricts to an isomorphism over $\P^1\setminus\delta(\caL)$ and its output is a Mori fiber space over $\P^1$.
\end{prop}

\begin{proof}
Put $X'=X_0$ and $f_0=\psi_\caL\circ\tau\:X_0\to\P^1$. Since general fibers of $f_0$  are Fano varieties, their canonical divisors are not pseudo-effective, hence $K_{X_0}$ is not pseudo-effective over $\P^1$. This implies that the output $\hat{f}=f_m\:\hat{X}\to\P^1$ of any  $K_{X_0}$-MMP $\varphi\:X_0\map \hat{X}$ relative to $ f_0\: X_0\to\P^1$ that terminates, factors through a Mori fiber space $g\:\hat{X}\to T$ over a normal projective variety $T$ given by the contraction of some extremal ray of $\overline{\NE}(\hat{X}/\P^1)$. The termination of at least one such relative MMP is guaranteed by \cite[Corollary 1.3.3]{BCHM}. 

We show by induction that for every $k\in\{1,\ldots,m\}$ the birational map $\varphi_k\: X_{k-1}\map X_k$ restricts to an isomorphism over $\P^1\setminus\delta(\caL)$. Let $\sigma\: X_{k-1}\to Y$ be the birational extremal contraction associated to some extremal ray in $\overline{\NE}( X_{k-1}/\P^1)$ and let $C$ be a curve contracted by $\sigma$. Since the extremal ray lies in $\overline{\NE}( X_{k-1}/\P^1)$, $C$ is contained in some fiber $f_{k-1}^*(p)$, $p\in \P^1$. Note that by the rigidity lemma \cite[Lemma 1.6]{KollarMori-Bir_geometry}, $\sigma$ does not contract $f_{k-1}^*(p)$.
By definition of a terminal rank one Fano pencil and of a compatible thrifty resolution, for every $p\in \P^1\setminus \delta(\caL)$ the fiber $f_0^*(p)$, and hence by induction the fiber $f_{k-1}^*(p)$ endowed with its reduced structure is a projective variety of Picard rank one.  By Remark \ref{rem:trivial contractions} the restriction of $\sigma$ to $f_{k-1}^*(p)$ is an isomorphism. It follows that the exceptional locus of $\sigma$ is contained in fibers of $f_{k-1}$ over $\delta(\caL)$, hence that $\varphi_k$ restricts to an isomorphism over $\P^1\setminus\delta(\caL)$. Finally, since general fibers of $\hat{f}\:\hat{X}\to\P^1$ have Picard rank one, $\hat{f}\:\hat{X}\to\P^1$ cannot be decomposed into a Mori fiber space $g\:\hat{X}\to T$ over a base $T\to\P^1$ of positive relative dimension, hence $\hat{f}\:\hat{X}\to\P^1$ itself is a Mori fiber space. 
\end{proof}

\begin{cor}\label{cor:MFS-Completion} Let $X$ be a normal projective variety and let $\caL$ be a terminal rank one Fano pencil on $X$ that admits a compatible thrifty resolution. Then $X_{\delta(\caL)}\setminus \Bs\caL$ (see Notation \ref{nota:open-subsets}) admits a Mori fiber completion $\pi\:V\to \P^1$. 
Furthermore if $p\in \P^1\setminus \delta(\caL)$ then the scheme-theoretic fiber $\pi^*(p)$ is isomorphic to the member $\caL_p$ of $\caL$.  
\end{cor}

\begin{proof} Let $\tau\:X'\to X$ be a compatible thrifty resolution of $\caL$. Put $\tau'=\Gamma(\tau)$ and let  $\varphi\:X'\map V$ be a  $K_{X'}$-MMP relative to $\pp\circ\ \tau'\:X'\to \P^1$ which terminates. By Proposition \ref{thm:MainThm-Completion}, $V$ has a structure of a Mori fiber space $\pi\:V\to \P^1$. The desired open embedding is given by the restriction to $X_{\delta(\caL)}\setminus \Bs \caL$ of the birational map $\varphi\circ (\gamma\circ \tau')^{-1}\:X\map V$. Indeed, the birational map $\gamma^{-1}\:X\map \Gamma$ induces an isomorphism between $X_{\delta(\caL)}\setminus \Bs \caL$ and $\Gamma_{\delta(\caL)}\setminus E$. On the other hand, since by the definition of a compatible thrifty resolution the image of $\Exc(\tau')$ by $\pp\circ\tau'\:X'\to \Gamma\to \P^1$  is contained in $\delta(\caL)$, the birational map $(\tau')^{-1}$ restricts to an isomorphism over $\pp^{-1}(\P^1\setminus \delta(\caL))=\Gamma_{\delta(\caL)}$. It follows in turn from Proposition \ref{thm:MainThm-Completion} that the rational map $\varphi\circ (\tau')^{-1}\:\Gamma\map V$ restricts to an isomorphism over $\Gamma_{\delta(\caL)}$. The second assertion follows from Proposition \ref{lem:graph-reso-members}. 
\end{proof}

The next example illustrates the process of taking a compatible thrifty resolution of a terminal rank one Fano pencil and then running a relative MMP as above. It shows in particular that different runs of the MMP may lead to different Mori fiber completions. 

\begin{example}[Mori fiber completions of $\A^2$ from pencils]\label{ex:conic-pencil}
Let $\caL$ be a pencil on $\P^2=\Proj(\C[x,y,z])$ generated by a smooth conic $C$ and twice a line $H$ tangent to $C$. Up to a projective equivalence we may assume that $C=\{xz-y^2=0\}$ and $H=\{x=0\}$. The graph of $\psi_\caL$ is the surface $\Gamma \subseteq \P^2\times \P^1_{[u:v]}$ defined by the bi-homogeneous equation $(xz-y^2)v-x^2u=0$. It is normal and its unique singular point $p=([0:0:1],[1:0])$ is supported at the intersection of the proper transform of $H$ with the exceptional divisor $E\cong \P^1$ of $\gamma=\pr_1\:\Gamma \to \P^2$. The singular point is a cyclic quotient singularity of type $A_3$. Let $\tau'\:X'\to \Gamma$ be the minimal resolution of the singularity of $\Gamma$. Then $\tau=\gamma\circ \tau'$ is a compatible thrifty resolution of $\caL$. The exceptional locus of $\tau'$ is a chain of three smooth rational curves $F_0$, $F$, $F_1$ with self-intersection numbers equal to $-2$, having $F$ as its middle component. The proper transforms $E'$ and $H'$ of $E$ and $H$ in $X'$ are smooth rational curves with self-intersection $-1$, and they intersect $\Exc(\tau')$ transversally along the curves $F_0$ and $F$ respectively.

The $K_{X'}$-MMP relative to $f'=\pp\circ\ \tau'\:X'\to \P^1$ first contracts $H'$, then the image of $F$ and finally either the image of $F_0$ or the image of $F_1$. The resulting Mori fiber space $\pi\:V\to \P^1$ is thus isomorphic to $\rho_0=\pr_2\:\F_0=\P^1\times \P^1 \to \P^1$ in the first case and to the Hirzebruch surface $\rho_1\:\F_1\to \P^1$ in the second case. The first case yields an open embedding of $\A^2=\P^2\setminus H$ into $\F_0$ as the complement of the proper transforms of $E$ and $F_1$, which are respectively a section with self-intersection number $0$ and a fiber of $\rho_0$. In the second case we obtain an open embedding of $\A^2$ into $\F_1$ as the complement of the proper transforms of $E$ and $F_0$, which are respectively the negative section and a fiber of $\rho_1$.
\end{example}

\smallskip \section{Obtaining Mori fiber completions from special pencils}\label{sec:H-special}

We now consider a class of pencils to which the methods of Section \ref{sec:pencils_and_resolutions} apply. A \emph{polarized ($\Q$-factorial) pair} $(X,H)$ is by definition a pair consisting of a normal projective ($\Q$-factorial) variety $X$ of dimension at least $2$ and an ample prime Weil divisor $H$ on $X$. 

\begin{defn}[$H$-special pencils]\label{def:H-special_pencil} Let $(X,H)$ be a polarized $\Q$-factorial pair. An \emph{$H$-special} pencil on $X$ is a pencil $\caL$  which satisfies the following properties:
\begin{enumerate}[(a)]
\item $dH$ is a member of $\caL$ for some integer $d\geq 1$.
\item The base locus $\Bs \caL$ is irreducible.  
\item If $d=1$ then the base scheme $\Bs \caL$ is smooth or its support is contained in $\Sing(H)$.
\end{enumerate}
\end{defn}

\subsection{Integrity of members} 

\begin{lem}\label{lem:pencil-local-smooth}
Let $(R,\fm)$ be a noetherian integral local ring and let $f\in \fm$ be a nonzero element such that the ring $R/(f)$ is regular. Then $R$ is regular and for every $h\in \fm^2$ the ring $R/(f+h)$ is regular. 
\end{lem}
\begin{proof} 
Let $\pi\: R\to R/(f)$ be the quotient morphism. Since $R/(f)$ is regular, $f\notin \fm^2$ and the maximal ideal $\pi(\fm)$ is generated by a regular sequence $\pi(a_1),\ldots,\pi(a_n)$, where $a_i\in \fm$ and $n=\dim R/(f)$. It follows that $\fm$ is generated by the regular sequence $f,a_1,\ldots, a_n$, hence that $R$ is regular.   Furthermore,  $\fm^2\subseteq (f^2)+(a_1,\ldots,a_n)$, so for some $v\in R$ we have $h-vf^2\in (a_1,\ldots,a_n)$. Since elements in $1+\fm$ are invertible in $R$, we get $(f+h,a_1,\ldots,a_n)=(f(1+vf),a_1,\ldots,a_n)=\fm$. It follows that the images of $a_1,\ldots, a_n$ under the quotient homomorphism $R\to R/(f+h)$ form a regular sequence which and they generate the maximal ideal of $R/(f+h)$. Thus, $R/(f+h)$ is regular.
\end{proof}

\begin{lem}[Integrity and smoothness of members]\label{lem:integral-members}
Let $\caL$ be an $H$-special pencil on a polarized $\Q$-factorial pair $(X,H)$. Assume that some member of $\caL$ is smooth and Cartier at some point $x\in \Bs \caL$ and that in case $H$ is a member of $\caL$, $x \in \Sing H$. Then every member of $\caL$ not supported on $\Supp H$ is a prime divisor which is smooth and Cartier at $x$. 
\end{lem}

\begin{proof}
By assumption $\caL$ has a member $dH$ for some positive integer $d$ and a member $F\neq dH$ which is smooth and Cartier at $x\in S=(\Bs \caL)_\redd=(F\cap H)_\redd $. Any two members of $\caL$ differ by a principal divisor, so we infer that all members of $\caL$ are Cartier at $x$. Let $\fm$ be the maximal ideal of the local ring $\cO_{X,x}$ and let $f,h\in \fm$ be generators of the ideals of $F$ and $dH$ in $\cO_{X,x}$, respectively. Since $F$ is smooth at $x$, by Lemma \ref{lem:pencil-local-smooth}, $f\in\fm\setminus \fm^2$ and $\cO_{X,x}$ is regular. If $d>1$ then $h\in \fm^2$. If $d=1$ then by assumption $x$ is a singular point of $H$, so by the Jacobian criterion in $\cO_{X,x}$ the residue class of $h$ in $\fm/\fm^2$ is trivial, hence again $h\in \fm^2$.

The variety $X$ is projective and the divisor $H$ is ample, so since $F$ is $\Q$-Cartier, $S$ is a closed subset of $\Supp(H)$ of pure codimension $1$. Let $Y$ be any member of $\caL$ other than $dH$. Write $Y=\sum_{i=1}^{r}D_{i}$, where $D_i$ are prime divisors. Each $D_i$ is $\Q$-Cartier and $H$ is ample, so $(D_i\cap H)_\redd$ has pure codimension $1$ in $H$. But $(D_i\cap H)_\redd\subseteq (Y\cap H)_\redd=S$ and $S$ is irreducible, so $(D_i\cap H)_\redd=S$ for each $i$. The ideal of $D_i$ in $\cO_{X,x}$ is thus contained in $\fm$ for every $1\leq i\leq r$, and hence the ideal of $Y$ is contained in $\fm^r$. On the other hand, the ideal of $Y$ in $\cO_{X,x}$ is generated by $f+th$ for some $t\in \C$. Since $f\in\fm \setminus \fm^2$ and $h\in \fm^2$, we have $f+th\in \fm\setminus \fm^2$. Thus $r=1$ and so $Y$ is a prime divisor, which by Lemma \ref{lem:pencil-local-smooth} is smooth at $x$. 
\end{proof}

In Lemma \ref{lem:integral-members} the assumption that the base locus of $\caL$ is irreducible and that $\caL$ has a member which is smooth at some point of $\Bs\caL$ are both necessary to ensure primeness of all members of $\caL$ other than $dH$. Similarly for the additional assumption in case $d=1$, where $H$ is itself a member of $\caL$, that $\Supp(\Bs\caL)$ contains a singular point of $H$. This is illustrated by the following examples.

\begin{example}[Failure of integrity] \label{ex:irreducible}\
\begin{enumerate}[(a)] 
\item Let $\caL$ be a pencil on $\P^2$ generated by a smooth conic $C$ and twice a line $H$ meeting $C$ at two distinct points. The sum of two lines tangent to $C$ at those points is a reducible member of $\caL$. Similarly, the pencil in $\P^3$ generated by the projective cones over $C$ and $2H$ has a reducible member consisting of the union of two planes. In the latter case the base locus is reducible but connected.
\item Let $\caL$ be a pencil on $\P^2$ generated by a nodal cubic curve $C$ and $3H$, where $H$ is the line tangent to one of the two branches of $C$ at its singular point $p=(\Bs\caL)_\redd$. Then every member of the pencil is singular at $p$ and $\caL$ has a reducible member. Indeed, up to a projective equivalence we have $C=\{y^2z=x^2(x-z)\}$, $H=\{x=0\}$ and $p=[0:0:1]$. Then the members of $\caL$ are $\caL_{[a:b]}=\{a(x^2+y^2)z=bx^3\}$ and $\caL_{[1:0]}$ is the union of three distinct lines through $p$.
\item Let $\caL$ be a pencil on $\P^2$ generated a smooth conic $H$ and some other smooth conic meeting $H$ in one point $p$ only. The pencil has a non-reduced member supported on the line tangent to both conics at $p$.
\end{enumerate}
\end{example}

\begin{cor}\label{cor:H-special-members}
Let $\caL$ be an $H$-special pencil on a polarized $\Q$-factorial pair $(X,H)$. In case $H$ is a member of $\caL$ and $\Supp (\Bs \caL)\nsubseteq \Sing H$ assume additionally that $\Cl(X)=\Z\langle H\rangle$. If $\caL$ has a member which is smooth and Cartier on some neighborhood of $\Bs \caL$ then every member of $\caL$ other than $dH$ is a prime divisor which is smooth and Cartier on some neighborhood of $\Bs \caL$. \end{cor}
\begin{proof}
By assumption $\caL$ has $dH$ as a member for some positive integer $d$. Put $S=(\Bs \caL)_\redd$. We may assume that $d=1$ and $S\nsubseteq \Sing H$, as otherwise the corollary follows from Lemma \ref{lem:integral-members}. Then every member of $\caL$ is prime, because by assumption $\Cl(X)=\Z\langle H\rangle$ in this case. Moreover, by the definition of an $H$-special pencil (see Definition \ref{def:H-special_pencil}(c)), the base scheme $\Bs \caL$ is smooth, hence in particular reduced. Then for every  member $Y\neq H$ of $\caL$ we have $S=\Bs \caL=Y\cap H$ scheme-theoretically. Let $\fm$ be the maximal ideal of the local ring $\cO_{X,x}$ at a point $x\in S$. Then the ring $\cO_{S,x}$ is regular and isomorphic to $\cO_{X,x}/(y,h)= (\cO_{X,x}/(y))/(h)$ where $y$ and $h$ are generators of the ideals of $Y$ and $H$ in $\cO_{X,x}$, respectively. Since $Y\sim H$ is prime, the ring $\cO_{X,x}/(y)$ is integral and its quotient by $(h)$ is regular. So $\cO_{X,x}/(y)$ is regular by Lemma \ref{lem:pencil-local-smooth}, that is, $Y$ is smooth at $x$. 
\end{proof}

\subsection{$H$-special pencils of Cartier divisors}

Recall (Definition \ref{def:Fano}) that for a Fano variety $X$ for which $\Cl(X)\cong \Z\langle H\rangle $, where $H$ is an ample divisor, the \emph{index of $X$} is the unique integer $i_X$ for which  $-K_X\sim i_XH$.

\begin{lem} \label{lem:fat-member-take0}
Let $X$ be a Fano variety with the class group $\Cl(X)\cong \Z\langle H\rangle $ and let $Y\sim dH$ be a normal prime divisor on $X$. If $Y$ is Cartier in codimension $2$ and $d<i_X$ then $Y$ is a Fano variety. 
\end{lem}
\begin{proof}
Since $\Cl(X)\cong \Z$, $X$ is $\Q$-factorial, so $K_X+Y$ is $\Q$-Cartier. Since $Y$ is normal and Cartier in codimension $2$, the adjunction formula, Lemma \ref{lem:different}, implies that $-K_Y=-(K_X+Y)|_Y$ is an anticanonical divisor on $Y$. We have $i_X>d$, so the divisor $-(K_X+Y)\sim (i_X-d)H$ is ample, hence $-K_Y$ is ample. 
\end{proof}

It is more difficult to provide uniform conditions which ensure that a given member of an $H$-special pencil has Picard rank one. For pencils of Cartier divisors on mildly singular varieties we can rely on the following result  for Picard groups proven in \cite[Expos\'e XII, Corollary 3.6]{SGA2}.

\begin{lem}[Grothendieck-Lefschetz theorem] \label{lem:Grothencieck-Lefschetz} 
Let $Y$ be an ample effective Cartier divisor on a normal variety $X$. Assume that $H^i(Y,\cO_X(-\ell Y)|_Y)=0$ for $i=1,2$ and every $\ell>0$, and that $X\setminus Y$ is a local complete intersection. Then the restriction homomorphism $\Pic X\to \Pic Y$ is an isomorphism.
\end{lem}

\begin{cor}[Finding terminal Fano pencils of rank one]\label{prop:FanoHyperPencil-base-Take2} Let $X$ be a Fano variety of dimension at least $4$ whose singularities are rational and whose class group is generated by a prime Weil divisor $H$. Let $d\in \{1,2,\ldots,i_X-1\}$ and let $\caL\subseteq  H^0(X, \cO_X(dH))$ be an $H$-special pencil of Cartier divisors such that $X\setminus \Bs \caL$ is a local complete intersection and which has a  
terminal member smooth in a neighborhood of $\Bs \caL$. Then $\caL$ is a terminal rank one Fano pencil and every member of $\caL$ other than $dH$ is non-degenerate. 
\end{cor}

\begin{proof}
Since $H$ generates $\Cl(X)$, $X$ is in particular $\Q$-factorial and $H$ is ample, so the pair $(X,H)$ is polarized $\Q$-factorial. Since $\caL$ has a terminal member and $d<i_X$, general members of $\caL$ are terminal Fano varieties by Lemma \ref{lem:Q-term-crit} and Lemma \ref{lem:fat-member-take0}. By Corollary \ref{cor:H-special-members} every member $Y$ of $\caL$ other than $dH$ is prime. By assumption the divisor $Y$ is Cartier and, since  $\Cl(X)\cong\Z$, it is necessarily ample. Then $\cO_X(Y)$ is an ample invertible sheaf, so we have exact sequences $$0\to \cO_X(-(\ell+1)Y)\to \cO_{X}(-\ell Y) \to \cO_{X}(-\ell Y)|_{Y} \to 0 \quad \textrm{for every } \ell>0.$$ Since $\cO_X(Y)$ is ample, by \cite[Corollary 7.67]{Vanishing-Theorems-Cplx-Manifolds} (see also Lemma \ref{lem:rational_sing}(c)), we have $H^{i}(X,\cO_{X}(-\ell Y))=0$ for every $i\leq \dim X-1$ and $\ell>0$. Since $\dim X\geq 4$, the associated long exact sequence of cohomology gives $H^{i}(Y,\cO_X(-\ell Y)|_Y)=0$ for every $i=1,2$ and $\ell>0$. Since $X\setminus Y$ is contained in $X\setminus \Bs \caL$, Lemma \ref{lem:Grothencieck-Lefschetz} implies that $\Pic(Y)\cong \Pic(X)\cong \Z$.
\end{proof}

Combining the above results we obtain the following theorem.

\begin{thm}[Mori fiber completions from $H$-special pencils]\label{thm:FanoHyperPencil-2}
Let $X$ be a Fano variety of dimension at least $4$ whose singularities are rational and whose class group is generated by a prime Weil divisor $H$. Assume that $X\setminus H$ is terminal and that for some $d\in \{1,2,\ldots,i_X-1\}$ there exists a Cartier divisor $F\sim dH$ other than $dH$ for which the following hold:
\begin{enumerate}[(a)]
\item $F$ is terminal,
\item $F\cap H$ is irreducible and contained in $F_\reg$, 
\item $X\setminus (F\cap H)$ is a local complete intersection,  
\item If $d=1$ then $F\cap H$ is either a smooth  scheme or its support is contained in $\Sing(H)$.
\end{enumerate}
Then $X\setminus H$ admits a Mori fiber completion $\pi\:V\to \P^1$ such that all members of the pencil $\caL=\langle F,dH \rangle$ other than $dH$ appear as fibers of $\pi$. 
\end{thm}

\begin{proof} 
The pair $(X,H)$ is a  polarized $\Q$-factorial pair. Since $F$ is Cartier, the assumptions (b) and (d) imply that $\caL$ is an $H$-special pencil and on some neighborhood of $\Bs \caL$ the divisor $F$ is smooth, in particular $\Q$-factorial. Let $\psi_\caL\:X\map \P^1$ be the dominant rational map determined by $\caL$. By Corollary \ref{prop:FanoHyperPencil-base-Take2}, $\caL$ is a terminal rank one Fano pencil with degeneracy locus contained in $\{ (\psi_\caL)_*H \}$. Since by Corollary \ref{cor:H-special-members} every member of $\caL$ other than $dH$ is smooth in a neighborhood of $\Supp(\Bs \caL)$, property $(\mathbf{TQ}_\delta)$ (see Notation \ref{nota:open-subsets}) holds for $\delta=\{ (\psi_\caL)_*H \}$. Since $X\setminus H$ is $\Q$-factorial and terminal by assumption, Corollary \ref{prop:controled-discrepancy} implies that every thrifty $\Q$-factorial terminal resolution of $\caL$ is compatible, that is, its degeneracy locus is contained in $\delta$. The assertion then follows from Corollary \ref{cor:MFS-Completion}.
\end{proof}

As a corollary we obtain Mori fiber completions of affine varieties of dimension $\geq 4$ whose general fibers are completions of affine Fano varieties in the sense of Definition \ref{def:Affine-Fano}.

\begin{cor}[Affine Fano fibers] \label{cor:affine_Fano} In the setting of Theorem \ref{thm:FanoHyperPencil-2} assume further that $F$ is $\Q$-factorial, that $(F\cap H)_\redd$ is klt and that $d\leq i_X-2$. Let $\pi\:V\to \P^1$ be the Mori fiber completion of the affine variety $U=X\setminus H$ associated to the pencil $\caL=\langle F,dH \rangle$. Then the general fibers of $\pi|_U\:U\to \P^1$ are affine Fano varieties.
\end{cor}

\begin{proof} Let $B=V\setminus U$. For a general point $p\in \P^1$, let $V_p=\pi^*p$ and let $B_p$ denote the reduction of the restriction of $B$ to $V_p$ as a Weil divisor. By Lemma \ref{lem:Q-term-crit}, $V_p$ is a $\Q$-factorial terminal Fano variety of Picard rank one. On the other hand, it follows from the proof of Corollary \ref{cor:MFS-Completion} that the log pair $(V_p,B_p)$ is isomorphic to the log pair $(\caL_p, (\caL_p\cap H)_\redd)$. Since $(\caL_p\cap H)_\redd=(F\cap H)_\redd$ is irreducible and klt, the log pair $(V_p,B_p)$ is plt by Lemma \ref{lem:different}. We have $-(K_{\caL_p}+({\caL_p}\cap H)_\redd)=(-i_X+d) H|_{\caL_p}+ H|_{\caL_p}$ by adjunction, so $-(K_{V_p}+B_p)$ is ample.
\end{proof}

\smallskip \section{Mori fiber completions of affine spaces over $\P^1$}\label{sec:Affine-Spaces}

We now apply our results to the construction of $\Q$-factorial terminal Mori fiber completions of affine spaces over $\P^1$. We begin with a review of some known examples.

\subsection{Some known examples}

\begin{example}[Examples of product type]\label{ex:products}
For every $n\geq 1$ and every $\Q$-factorial terminal Fano variety $X$ of Picard rank one which is completion of $\A^n$, the projection $\pi=\pr_2\:V=X\times \P^1\to \P^1$ is a Mori fiber completion of $\A^n\times \A^1$. For instance, we can take $X=\P^n$, which is the only possibility for $n=1,2$. For $n=3$, smooth Fano threefolds of Picard rank one  which are completions of $\A^3$ have been classified by Furushima \cite{Fu93} (see also \cite{Pr}). These are: $\P^3$, the smooth quadric threefold in $\P^4$, the quintic del Pezzo threefold in $\P^6$, and a four dimensional family of prime Fano threefolds of genus $12$. In higher dimensions, other  examples of smooth Fano completions of $\A^n$ of Picard rank one are the quintic del Pezzo fourfold \cite[Theorem 3.1]{Pr93} in $\P^7$ and Fano-Mukai fourfolds of genus $10$ \cite{PZ-genus10_4folds_and_A4}.
\end{example}

There are several examples of Mori fiber completions of $\A^3$ which are not of product type as in Example \ref{ex:products}. For instance, for every $d\in\{1,2,3,4,5,6,8,9\}$ there exists a Mori fiber completion $\pi\:V\to \P^1$ of $\A^3$, whose general fibers are smooth del Pezzo surfaces of degree $d$. For $d=9$ we have only the locally trivial $\P^2$-bundles over $\P^1$.  A classical example with $d=8$ is recalled below. An example with $d=7$ does not exist, because the generic fiber of the corresponding del Pezzo fibration $\pi\:V\to \P^1$ would be a minimal smooth del Pezzo surface of degree $7$ and such surfaces do not exist over a field of characteristic zero, see \cite[Theorem 29.4]{Manin-cubic_forms}. An example with $d=6$ can be found in \cite[Theorem 1.2]{Pr16} and \cite{Fukuoka19}. A construction for $d=5$ is given in Example \ref{ex:dP5} below.  For examples with $d=1,2,3,4$ see \cite[Theorem 2]{DK4}. For $d\leq 6$ general fibers of the restriction of $\pi\:V\to \P^1$ to $\A^3$ are not isomorphic to $\A^2$. Indeed, otherwise the generic fiber of $\pi\:V\to \P^1$ would be a minimal smooth del Pezzo surface of degree $d\leq 6$ over the function field $\C(\P^1)$, containing a Zariski open subset isomorphic to $\A^2_{\C(\P^1)}$, which is impossible by \cite[Proposition 13]{DK4}. In particular, none of these completions is of product type. Note also that for $d\neq 9$ general fibers of the corresponding Mori fiber spaces are smooth del Pezzo surfaces of Picard rank higher than one, hence are not associated to any terminal rank one Fano pencil (cf.\ Example \ref{ex:low-dim-QT-FanoRk1}).

\begin{example}[A completion of $\A^3$ into a del Pezzo fibration of degree $8$] \label{ex:dP_fibr} 
Let $Q\subseteq \P^4$ be a smooth quadric threefold, let $H$ be a hyperplane section of $Q$ cut by a tangent hyperplane and let $F$ be a smooth hyperplane section of $Q$ such that the scheme-theoretic intersection $C=F\cap H$ is irreducible. Then $H$ is the quadric cone $H\cong \P(1,1,2)$, $F\cong \P^1\times \P^1$ and $C$ is a smooth rational curve. Let $\caL$ be a pencil on $Q$ generated by $F$ and $H$. For a general member $Y$ of $\caL$ the pair $(Y,C)$  is isomorphic to the pair consisting of $\P^1\times \P^1$ embedded as a smooth quadric surface in $\P^3$ and a smooth hyperplane section of it. Let $\alpha\:\widetilde{Q}\to Q$ be the blow-up of the unique singular point $q$ of $H$. Its exceptional divisor is $E\cong \P^2$. Let $\beta\:\widetilde{Q}\to \P^3$ be the contraction of the proper transform of $H$ onto a smooth conic $C'$. The latter is contained in $H_\infty=\beta(E)$, which is a hyperplane of $\P^3$. Let $\tau\:V\to Q$ be the blow-up of $Q$ with center at $C$ and exceptional divisor $D$. We then have a diagram of Sarkisov links 
\[\xymatrix{ & (\widetilde{Q},\alpha_*^{-1}H+E) \ar[dl]_{\beta} \ar[dr]^{\alpha}  & & & (V,D+\tau_*^{-1}H) \ar[dll]_{\tau} \ar[d]^{\pi} 
\\ (\P^3,H_\8)  & & (Q,H) \ar@{-->}[ll]^{\varphi} \ar@{-->}[rr]_{\Psi_\caL} & & \P^1}\]  
where $\pi\:V\to \P^1$ is a del Pezzo fibration of degree $8$. 

Since the hyperplane cutting $H$ is tangent to $Q$, we have $Q\setminus H\cong \A^3$. The variety $V$ is smooth and contains $\A^3$ as the complement of the total transform of $H$. General fibers of $\pi\:V\to \P^1$ have Picard rank $2$. On the other hand, the generic fiber $V_\eta$ of $\pi$ is a smooth quadric surface over the field $\C(\P^1)$ and we have $\rho(V_\eta)=\rho(V/\P^1)=1$. The Picard group of $V_\eta$ is generated by the restriction $D_\eta$ of $D$. We note that since $-(K_{V_\eta}+D_\eta)=D_\eta$ is ample, for every open subset $U\subseteq \P^1$ the open subset $\pi^{-1}(U)\setminus D\subseteq V$ is a relative affine Fano variety over $U$.

The birational map $\varphi=\beta\circ \alpha^{-1}\:Q\dashrightarrow \P^3$ is induced by the linear projection form the point $q\in \P^4$. It restricts to an isomorphism $Q\setminus H\cong \P^3\setminus H_\infty$. The composition $\Psi_\caL\circ \varphi^{-1}\:\P^3\dashrightarrow \P^1$ is given by the pencil $\caL'$ on $\P^3$ generated by $2H_\infty$ and the proper transform $F'\cong \P^1\times \P^1$ of $F$, which is a smooth quadric surface in $\P^3$ intersecting $H_\infty$ along the conic $C'$. 
\end{example}
 
Recall \cite{Fujita81} that the quintic del Pezzo fourfold  $W_5\subset \P^7$ is the intersection of the Grassmannian $\Gr(2,5) \subset \P^9$ with a general linear subspace of codimension $2$. In particular, $\Pic(W_5)\cong \Z\langle H\rangle $, where $H$ is a hyperplane section of $W_5$. It is known that $W_5$ is the unique smooth Fano fourfold with Fano index $i_{W_5}=3$ and $\Pic(W_5) \cong \Z$ generated by an ample generator $H$ such that $H^4 =5$. 

A general hyperplane section of $W_5$, called the quintic del Pezzo threefold, is a smooth Fano threefold $V_5\subseteq \P^6$ with $\Pic(V_5) \cong \Z\langle H_\8\rangle$, Fano index $i_{V_5}=2$ and $H_\8^3 =5$, where $H_\8$ is a hyperplane section of $V_5$.  Again, by \cite{Fujita81} this is the unique Fano threefold with these invariants. By \cite[Theorem A]{Fu93} $V_5$ has a normal hyperplane section $H$ with a unique singular point $p\in H$ of type $A_4$ such that $V_5\setminus H\cong \A^3$. By \cite[Proposition 15]{Fu86} $H$ contains a unique line $L\subseteq \P^6$. The line passes through $p$.

\smallskip
By blowing-up $V_5$ with the center being a suitably chosen anti-canonical curve $C$ in $H$, we now construct completions of $\A^3$ into $\Q$-factorial terminal threefolds with del Pezzo fibrations of degree $5$. In case of a smooth $C$ the construction was communicated to us by Masaru Nagaoka and in case of nodal $C$ it was suggested by a referee.
 
\begin{example}[Completions of $\A^3$ into del Pezzo fibrations of degree $5$]\label{ex:dP5} Let $H\subset V_5$ be a normal hyperplane section with a unique singular point $p\in H$ of type $A_4$ such that $V_5\setminus H\cong \A^3$ and let $L\subseteq \P^6$ be the unique line on $H$. The complement $H\setminus L$ is isomorphic to $\A^2$, so $\Cl(H)\cong \langle L\rangle\cong \Z$. Since $i_{v_5}=2$, by adjunction $\cO_H(-K_H)\cong \cO_{V_5}(H)|_H$ (see Lemma \ref{lem:different}), so $H$ is a singular del Pezzo surface of degree ${(-K_H)}^2=H^3=5$. Furthermore, given a hyperplane section $F$ of $V_5$ not containing $L$, we have $(-K_H)\cdot L=F|_H\cdot L=F\cdot L=1$. So $-K_H\sim 5L$ and hence, $\Pic(H)\subset \Cl(H)$ is the subgroup of index $5$ generated by the class of $-K_H$.

Every divisor $C$ in the complete linear system $|-K_{H}|$ appears as the base locus of a unique pencil $\caL_C$ on $V_5$ containing $H$ as a member. Indeed, since $V_5$ is a smooth Fano threefold, we have $H^1(V_5,\cO_{V_5})\cong 0$ and the assertion follows from long exact sequence of cohomology associated with the short exact sequence of sheaves $$0\to \cO_{V_5}\to\cO_{V_5}(H)\to\cO_H(-K_H)\to 0.$$ By adjunction general members of $\caL_C$ are del Pezzo surfaces of degree $5$. Denote the blowup of $V_5$ with center $C$ by $\tau\:V_C\to V_5$. Let $E$ be the exceptional divisor of $\tau$ (possibly reducible and non-reduced). The rational map $\psi_{\caL_C}\:V_5\map \P^1$ lifts to a fibration $\pi\:V_C\to \P^1$ whose general fiber is a del Pezzo surface of degree $5$ and $V_C$ contains $\A^3$ as the complement of the total transform of $H$. The class group of $V_C$ is generated be the classes of irreducible components of $E$ and of the proper transform $\tau_*^{-1}H$, which is a fiber of $\pi$. Singularities and the $\Q$-factoriality of $V_C$ depend on the properties of the chosen anticanonical divisor $C$.  

Let us briefly recall some elements concerning the geometry of anti-canonical divisors on $H$. We let $\gamma\:S\to H$ be the minimal resolution of singularities. We have $K_S\sim \gamma^*K_H$ and the exceptional locus $\Exc \gamma$ is a chain $E_1+E_2+E_3+E_4$ of $(-2)$-curves. It is known that the proper transform  $L'$ of the unique line $L$ on $H$ is a $(-1)$-curve meeting $\Exc \gamma$ only once and normally at a point of $E_3$. There exists a birational morphism $\sigma\:S\to \P^2$ which contracts $L'\cup E_3\cup E_2 \cup E_1$ and maps $E_4$ onto a line $\ell$. 
We have $\sigma^*K_{\P^2}\sim K_S-E_1-2E_2-3E_3-4L'$, so we obtain 
\begin{equation}\label{eq:K_H}
\sigma_*\gamma^*(-K_H)\sim -K_{\P^2}\text{\ \ and\ \ } (-K_H)\sim \gamma_*\sigma^*(-K_{\P^2}) -4L.
\end{equation}
Given $C\in |-K_H|$ we put $C'=\gamma_*^{-1}C$ and $\ov C=\sigma(C')$. Clearly, $\deg \ov C\leq 3$. Put $q=\sigma(L')$. Write $\sigma_*\gamma^*C=m\ell+D$, where $m\in\{0,1,2,3\}$ and $D$ is an effective divisor of degree $3-m$ not containing $\ell$ in its support. By \eqref{eq:K_H} $C=(3m-4)L+\gamma_*\sigma^*D$.

We now discuss the geometry of $V$ for some particular choices of $C$.

Consider the case $p\notin C$. Since $L$ is ample, the equality $L\cdot C=1$ implies that $C$ is irreducible and reduced. The induced morphisms $\gamma\:C'\to C$ and $\sigma\:C'\to \ov C$ are isomorphisms. By \eqref{eq:K_H} $\sigma_*\gamma^*C=\ov C$ is either an elliptic curve or a rational nodal curve or a rational cuspidal curve smooth at $q$ and intersecting $\ell$ with multiplicity $3$ at $q$. The fact that $H$ is smooth along $C$ implies that general fibers of $\pi\:V_C\to \P^1$ are smooth del Pezzo surfaces. Furthermore, since $C$ is irreducible and reduced, $E$ is irreducible and reduced. Since $E$ is Cartier, $\Pic(V_C)=\Cl(V_C)$ is freely generated by the classes of $E$ and $\tau_*^{-1}H$. In particular, $V_C$ is $\Q$-factorial and has Picard rank $2$. If $C$ is an elliptic curve then $V_C$ is smooth. If $C$ is nodal then locally analytically around the node the blowup of $C$ is isomorphic to the blowup of $\{xy=z=0\}\subseteq \A^3$, so its unique singular point is analytically isomorphic to $((0,0,0),[1:0])\in \{zv=xyu\}\subseteq \A^3\times \P^1$. Thus in this case $V_C$ has a unique ordinary double point supported on the inverse image by $\tau$ of the singular point of $C$. Finally, if $C$ is cuspidal then locally analytically around the cusp the blowup of $C$ is the blowup of $\{x^2+y^3=z=0\}\subseteq \A^3$. In this case $V_C$ has a unique compound du Val singularity $cA_2$ supported on the inverse image by $\tau$ of the singular point of $C$. In any case $V_C$ is a $\Q$-factorial terminal threefold and $\pi\:V_C\to \P^1$ is a del Pezzo fibration. 

Consider the case when $C$ is cut out by a general hyperplane section $F$ of $V_5$ passing through $p$. In particular, $F$ is smooth away from $p$ by Bertini's theorem and it does not contain $L$. The equality $F\cdot L=1$ implies that $F$ is smooth at $p$, hence $F$ is smooth. General fibers of $\pi\:V_C\to \P^1$ are thus smooth del Pezzo surfaces of degree $5$. By the choice of $F$, $L$ is not an irreducible component of $C$. In the notation as above, using the identity $C=(3m-4)L+\gamma_*\sigma^*D$ one checks that $m=1$ and that $D$ is a conic intersecting $\ell$ normally at $q$. By the generality assumption on $F$, $D$ is a smooth conic and then $C$ is an irreducible and reduced nodal rational curve smooth off $p$. So as in the previous situation, $V_C$ is a terminal $\Q$-factorial threefold with a unique ordinary double point supported on the inverse of $p$ by $\tau$ and $\pi\:V_C\to \P^1$ is a del Pezzo fibration. 
\end{example}
 
\begin{rem} In Example \ref{ex:dP5} one can easily work out all possible geometries of $C\in |-K_H|$ using the equality $C=(3m-4)L+\gamma_*\sigma^*D$. For instance, if $C$ does not contain $L$ in its support then one shows that $\ov C$ is a reduced conic, smooth or singular, meeting $\ell$ normally at $q$. Assume that it is singular. Then $C$ is a sum of two smooth rational curves intersecting at $p$ only and the corresponding threefold $V_C$ has a unique ordinary double point. The class group of $V_C$ is freely generated by classes of the two irreducible components of $E$ and the class of $\tau_*^{-1}H$ and $V_C$ is not $\Q$-factorial.
\end{rem}

\begin{example}[A completion of $\A^4$ into a quintic del Pezzo threefold fibration]\label{exa:delPezzofourfold-Take1} 
By \cite[Theorem 3.1(iv)]{Pr93} there exists an open embedding of $\A^4$ into the quintic del Pezzo fourfold $W_5$ such that the complement is a normal hyperplane section $H$ of $W_5$ (its singular locus consists of a unique ordinary double point $p$). Let $\caL$ be a pencil on $W_5$ generated by $H$ and by a general hyperplane section $F$. The base locus of $\caL$ is a smooth del Pezzo surface $S$ of degree $5$. Let $\tau\:V\to W_5$ be the blowup of $W_5$ with center at $S$. The rational map $\psi_\caL\:W_5\map \P^1$ lifts to a Mori fiber space $\pi\:V\to \P^1$ and its general fibers are quintic del Pezzo threefolds. The variety $V$ is smooth and it contains $\A^4$ as the complement of the union of the proper transform of $H$ and of the exceptional divisor $E$ of $\tau$. We note that $E$ intersects a general fiber $V_5$ of $\pi$ along a smooth del Pezzo surface $B$ of degree $5$. In particular, $V_5\setminus B$ is an affine Fano variety. Since $B$ is smooth, by the classification in \cite{Fu93}, $V_5\setminus B$ is not isomorphic to $\A^3$.

 We argue that $V_5\setminus B$ is not super-rigid (see Definition \ref{def:affine_rigid}). Let $\ell\subseteq B$ be a line and let $T\subset V_5$ be the surface swept out by the lines in $V_5$ intersecting $\ell$. The projection from 
 $\ell$ defines a birational map $\alpha\:V_5 \map Q$ to a smooth quadric threefold $Q$ in $\P^4$, which contracts $T$ onto a rational cubic contained in a hyperplane section $Q_0$ of $Q$. The image of $B$ by $\alpha$ is a smooth hyperplane section $Q_\infty $ of $Q$ and $\alpha$ induces an isomorphism $V_5\setminus (B\cup T)\cong Q\setminus (Q_0\cup Q_\infty)$. By Example \ref{ex:dP_fibr}, $Q\setminus (Q_0\cup Q_\infty)$ contains a relative affine Fano variety over $\P^1\setminus \{0,\infty\}$. So  $V_5\setminus B$ contains a relative affine Fano variety over a curve, hence is not super-rigid.
\end{example}

\subsection{Pencils on smooth Fano varieties}\label{sub:ComplSmooth}

Recall that the Grassmannian $\Gr(k,n)$, which parameterizes $k$-dimensional linear subspaces of a complex vector space of dimension $n$, is a smooth Fano variety of dimension $k(n-k)$ with class group isomorphic to $\Z$ and Fano index $i_{\Gr(k,n)}=n$ (see e.g. \cite[Lemma 10.1.1, p. 510]{Dolgachev-Classical_AG_modern_view}). It has a natural cover by affine open subsets isomorphic to $\A^{k(n-k)}$. Namely, denoting the vector space by $V$, we have the Pl\"ucker embedding $$\mathrm{pl}\:\Gr(k,n)\mono \P(\Lambda^k V)=\Proj(\C[\{x_I\}]),$$ where $I$ ranges through the set of subsets of $k$ distinct elements in $\{1,\ldots, n\}$, which associates to a closed point  $\Lambda\in \Gr(k,n)$, represented by a $k \times n$-matrix $A_\Lambda$ of rank $k$, the collection of the $k\times k$-minors of $A_\Lambda$. Then for every subset $I\subseteq \{1,\ldots, n\}$ of $k$ distinct elements the open subset $\Gr(k,n)\setminus\{x_I=0\}$ is isomorphic to $\A^{k(n-k)}$.

\begin{prop}[Mori fiber completions from pencils on Grassmannians] \label{prop:completion_from_pencils} For $k(n-k)\geq 4$, let $H$ be a hyperplane section of $\Gr(k,n)$ such that $\Gr(k,n)\setminus H \cong \A^{k(n-k)}$, let $d\in\{1,\ldots, n-1\}$ and let $F \subset \Gr(k,n)$ be an integral hypersurface such that $F\sim dH$. Assume that $S=F\cap H$ is irreducible and contained in the smooth locus of $F$ and that either $d\geq 2$ or $d=1$ and $S\subseteq \Sing(H)$. Then $\A^{k(n-k)}=\Gr(k,n)\setminus H$ admits a Mori fiber completion over $\P^1$ such that all members of the pencil $\langle F,dH\rangle$ other than $dH$ appear as fibers.
\end{prop}

\begin{proof} Since  $S=F\cap H$ is contained in the smooth locus of $F$, every member of $\caL$ other then $dH$ is smooth along $\Bs\caL$ by Lemma \ref{lem:pencil-local-smooth}. Since on the other hand $\Gr(k,n)$ is smooth, it follows from Bertini's theorem that a general member of $\caL$ is smooth away from $\Bs\caL$, hence smooth. Since $S$ is by assumption irreducible and contained in the smooth locus of $F$, the assertion follows from Theorem \ref{thm:FanoHyperPencil-2}. 
\end{proof}

We now deduce Theorem \ref{thm2}, which asserts that given $n\geq 2$ and a hyperplane in $H \subset \P^n$, for every integral hypersurface $F\subseteq \P^n$ of degree $d\leq n$ such that $F\cap H$ is irreducible and contained in the smooth locus $F_{\reg}$ of $F$ there exists a completion of the affine $n$-space $\A^n \cong \P^n \setminus H$ into a Mori fiber completion over $\P^1$ such that all members of the pencil $\langle F,dH\rangle$ other than $dH$ appear as fibers.

\medskip
\begin{proof}[Proof of Theorem \ref{thm2}]
The case $d=1$ is obvious, so we may assume that  $d\geq 2$. We have $\Gr(1,n+1)=\P^n$, so for $n\geq 4$ the result follows from Proposition \ref{prop:completion_from_pencils}. We are thus left with the three cases $(n,d)=(2,2)$, $(3,2)$ and $(3,3)$. The case $(2,2)$ is treated in  Example \ref{ex:conic-pencil}. In the case $(3,2)$, $F\subset \P^3$ is an integral quadric surface such that $F \cap H$ is irreducible and contained in $F_{\reg}$. By Corollary \ref{cor:H-special-members} and by Bertini's theorem, a general member of the pencil $\caL=\langle F,2H\rangle$ is a smooth quadric surface, so the assertion follows from Example \ref{ex:dP_fibr}. In the remaining case $(3,3)$, $F\subset \P^3$ is an integral cubic surface such that $F \cap H$ is irreducible and contained in $F_{\reg}$. Again, by Corollary \ref{cor:H-special-members} and Bertini's theorem, a general member of the pencil $\caL=\langle F, 3H \rangle$ is a smooth cubic surface which intersects $H$ along a smooth elliptic curve, so the result follows from \cite[Theorem. (a)]{DK3}. 
\end{proof}

\begin{example}[Families of Mori fiber completions of $\A^n$] Let $n\geq 4$ and $d\geq 2$ be integers. Put $\P^n=\Proj(\C[x_0,\ldots, x_n])$ and  $H=\{x_0=0\}\subset \P^n$. Let $\C[x_1,\ldots, x_n]_{\leq d}$ denote the affine space of polynomials of total degree at most $d$. Let $\cV_d$ denote its open subset consisting of polynomials $f$ of total degree precisely $d$ for which the scheme-theoretic intersection of $H$ with the closure $F$ in $\P^n$ of the zero locus of $f$ in $\A^n=\Spec(\C[x_1,\ldots, x_n])$ is smooth. So, for every $f\in \cV_d$, $F$ is an irreducible hypersurface of degree $d$ which contains  the smooth variety $F\cap H$ in its smooth locus. Theorem \ref{thm2} thus applies and for each $f\in \cV_d$ gives the existence of a Mori fiber completion $\pi\:X\to \P^1$ of $\A^n=\P^n\setminus H$ such that all members of the pencil $\langle F,dH\rangle$ other than $dH$ appear as fibers, and for which we have a commutative diagram \[\xymatrix{\A^n \ar[d]_{f} \ar[r] & X \ar[d]^{\pi} \\ \A^1 \ar[r] & \P^1}\] where the horizontal morphisms are open immersions. 
\end{example}

We now present two other examples of completions of $\A^4$ constructed respectively from the quintic del Pezzo fourfold, Fano-Mukai fourfolds of genus $10$ and the quintic del Pezzo fivefold. 

\begin{example}[A completion of $\A^4$ into a non-rational Fano threefold fibration] \label{exa:delPezzofourfold-Take 2} As in Example \ref{exa:delPezzofourfold-Take1} above, let $W_5\subset \P^7$ be the quintic del Pezzo fourfold and let $H\subset W_5$ be a singular hyperplane section with a unique ordinary double point $p$ whose complement is isomorphic to $\A^4$. Let $\caL$ be the pencil on $W_5$ generated by $2H$ and a general quadric section $F$.  A general member $Y$ of $\caL$ is a smooth Fano threefold of Picard rank one and index one isomorphic to the intersection of the Grassmannian $\Gr(2,5)$ with two hyperplanes and a quadric (family $B_{10}$ in \cite[Theorem 5.3]{Beau77}). Furthermore, $S=(\Bs\caL)_\redd=H|_Y$ is a smooth anticanonical divisor on $Y$, hence it is a smooth $\mathrm{K}3$ surface (which implies that $Y\setminus S$ is not an affine Fano variety). By Theorem \ref{thm:FanoHyperPencil-2} the pencil $\caL$ gives rise to a Mori fiber completion $\pi\:V\to \P^1$ of $\A^4=W_5 \setminus H$, whose general fibers are isomorphic to the general members of $\caL$. Thus, by \cite[Theorem 5.6(ii)]{Beau77}, general fibers of $\pi$ are non-rational. 
\end{example}

\begin{example}[A completion of $\A^4$ into a genus $10$ Fano threefold fibration]  \label{exa:mukai4-fold}
A Fano-Mukai fourfold of genus $10$ is a smooth Fano fourfold $X$ of Picard rank one, Fano index $i_X=2$, and genus $g:=\frac{1}{2}(H^4)+1=10$, where $H$  is an ample generator of $\Pic X$. By \cite[Remark 13.4, Theorem 1.1]{PZ-genus10_4folds_and_A4} the moduli space of such fourfolds has dimension one and for every such fourfold $X$ there exists an open embedding $\A^4\mono X$ whose complement is a generator $H_\8$ of $\Pic X$ and whose singular locus $T$ is a surface. 

Let $\caL$ be the pencil generated by $H_\8$ and a general smooth member $F$ of the complete linear system $|H_\8|$.  Since the restriction homomorphism $\Pic(X)\to \Pic(F)$ is an isomorphism by Lemma \ref{lem:Grothencieck-Lefschetz}, the base locus $\Bs \caL=F\cap H_\8$ of $\caL$ is an irreducible and reduced surface, which we denote by $S$. The singular locus of $S$ is equal to the curve $C=T \cap S$. Since $S$ is singular but not contained in the singular locus of $H_\infty$, we cannot directly apply Theorem \ref{thm:FanoHyperPencil-2} to obtain from $\caL$ a Mori fiber completion of $\A^4=X\setminus H_\8$ over $\P^1$. Instead, we argue as follows. By Bertini's theorem a general member $\caL_\lambda$ of $\caL$ is smooth off the base locus $S$. Since the scheme-theoretic intersection $\caL_\lambda|_{H_\8}=S$ is smooth off $C=T \cap S$, it follows that $\caL_\lambda$ is smooth off the curve $C$. Since $F$ is smooth at every point $x\in C$ whereas $H_\8$ is singular there, every member of $\caL$ other than $H_\8$ is smooth at $x$ by Lemma \ref{lem:pencil-local-smooth}. We infer that a general member $\caL_\lambda$ is smooth, so the base locus $S$ of $\caL$ is the scheme-theoretic transversal intersection of any two smooth members of $\caL$. The blow-up $\tau\: V \to X$ of $S$ is then a resolution of $\psi_\caL\:X\map \P^1$ and $\psi_\caL \circ \tau\: V \to \P^1$ is a Mori fiber space whose general fibers are Fano threefolds of Picard rank one, index $1$ and genus $10$. They are rational by \cite[Theorem 4.6.7]{Iskovskikh_Prokhorov-Fano_varieties} but are not completions of $\A^3$ by \cite{Fu93}. The  variety $V$ contains $\A^4$ as the complement of the union of the proper transform of $H_\infty$ and of the exceptional locus of $\tau$.
\end{example}

\begin{example}[Mori fiber completions of $\A^5$ from pencils on the quintic del Pezzo fivefold $Z_5$] \label{exa:delPezzofivefold} 
Recall \cite{Fujita81} that the quintic del Pezzo fivefold  $Z_5\subset \P^8$ is the intersection of the Grassmannian $\Gr(2,5) \subset \P^9$ with a general linear hyperplane. In particular, $\Pic(Z_5) \cong \Z\langle H\rangle$, where $H$ is a hyperplane section of $Z_5$. It is known that $Z_5$ is the unique smooth Fano fivefold with Fano index $i_{Z_5}=4$ and $\Pic(Z_5) \cong \Z$  generated by an ample generator $H$ for which $H^5 =5$. Furthermore, it follows for instance from the alternative description of $Z_5$ given in \cite[(7.10)]{Fujita81} that there exists an open embedding of $\A^5$ into $Z_5$ whose complement is a non-normal hyperplane section $H$ of $Z_5$. 

Let $\caL_3$ be a pencil on $Z_5$ generated by $3H$ and a general cubic section $F_3$.  A general member of $\caL_3$ is a smooth Fano fourfold of Picard rank one and Fano index one. By Theorem \ref{thm:FanoHyperPencil-2} the pencil $\caL_3$ gives rise to a Mori fiber completion $\pi_3\:X_3\to \P^1$ of $\A^5=Z_5 \setminus H$, whose general fibers are isomorphic to the general members of $\caL_3$. In a similar way, a pencil $\caL_2$ generated by $2H$ and a general quadric section $F_2$ gives rise to a Mori fiber completion $\pi_2\:X_2\to \P^1$ of $\A^5=Z_5 \setminus H$ whose general fibers are smooth Fano fourfold of Picard rank one and Fano index two. We do not know whether general fibers of these fibrations are rational.

Finally, one can consider a pencil $\caL_1$ generated by $H$ and a general hyperplane section $F_1$ of $Z_5$. The base locus of $\caL_1$ is an irreducible singular threefold $V$ whose singular locus is equal to a hyperplane section of the singular locus of $H$. Arguing as in Example \ref{exa:mukai4-fold}, we see that a general member of $\caL_1$ is smooth, so the base locus $V$ of $\caL_1$ is the scheme-theoretic transverse intersection of any two smooth members of $\caL_1$. The blow-up $\tau\: X_1 \to Z_5$ of $V$ is then a resolution of $\psi_{\caL_1}\:Z_5\map \P^1$ and $\psi_{\caL_1} \circ \tau\: X_1 \to \P^1$ is a Mori fiber space whose general fibers are quintic del Pezzo fourfolds $W_5$. The variety $X_1$ contains $\A^5$ as the complement of the union of the proper transform of $H$ and of the exceptional locus of $\tau$. A general fiber of the restriction of $\psi_{\caL_1} \circ \tau$ to $\A^5$ has a completion into $W_5$ with a smooth hyperplane section of $W_5$ as a boundary, hence by \cite[Theorem 3.1]{Pr93} is not isomorphic to $\A^4$ (even though $W_5$ is a completion of $\A^4$).
\end{example}

\begin{example}[Mori fiber completions of $\A^n$ with super-rigid affine Fano general fibers]\label{ex:affine_super-rigid}
Let $n\geq 4$ and let $\caL$ be a pencil on $\P^{n}$ generated by a general hypersurface of degree $n-1$ and by $(n-1)H$, where $H$ is a hyperplane. By Corollary \ref{cor:affine_Fano}, $\caL$ gives rise to a Mori fiber completion $\pi\:V\to \P^1$ of $\A^n=\P^n\setminus H$  such that general fibers of $\pi|_{\A^n}$ are smooth affine Fano varieties, isomorphic to the complement of a smooth hyperplane section of a hypersurface of degree $n-1$ in $\P^n$. If $n\geq 6$ then it is known that such affine Fano varieties are super-rigid \cite[Theorem 2.8, Example 2.9]{CDP17}. For $n=4,5$ the super-rigidity of  complements of general hyperplane sections of respectively smooth cubic threefolds in $\P^4$ and smooth quartic fourfolds in $\P^5$ is an open problem. 
 \end{example}

\subsection{Pencils on weighted projective spaces}\label{sec:Cone_construction}

An important class of $\Q$-factorial rational Fano varieties consists of weighted projective spaces. We fix notation and summarize some basic facts (see e.g. \cite{Dolg-WPS} and \cite{Fl00} for more). Given a non-decreasing sequence of positive integers $\bar a=(a_0,\ldots,a_n)$, we define an $\N$-grading on $\C[x_0,\ldots,x_n]$ by putting $\deg x_i=a_i$, and we let $\P(\bar a)= \Proj(\C[x_0,\ldots,x_n])$. The inclusion of graded rings $\C[y_0^{a_0},\ldots,y_n^{a_n}]\subseteq \C[y_0,\ldots,y_n]$ leads under the identification $x_i=y_i^{a_i}$ to a finite  morphism 
\begin{equation}\label{eq:pi}
\pi\:\P^n\to\P(\bar a)\cong\P^n/(\Z_{a_0}\times\cdots\times \Z_{a_n}),
\end{equation} where the action is diagonal, by multiplication by an $a_i$-th root of unity on the $i$-th factor, see \cite[\S 1.2.2]{Dolg-WPS}. The variety $\P(\bar a)$ is covered by the affine open subsets $\{x_i\neq 0\}\cong \A^n/\Z_{a_i}$, where the generator $\varepsilon$, a primitive  $a_i$-th root of unity, acts by $$(x_0,\ldots,\widehat{x_i}\ldots,x_n)\mapsto (\varepsilon^{a_0} x_0,\ldots,\widehat{ \varepsilon^{a_i}x_i}\ldots,\varepsilon^{a_n}x_n).$$ In particular, $\P(\bar a)$ is normal and $\Q$-factorial \cite[Lemma 5.16]{KollarMori-Bir_geometry}, with finite quotient singularities. (For a criterion when $\P(\bar a)$ is klt see \cite[Proposition 2.3]{Kasprzyk-terminal_WPS}, cf.\ \cite[Proposition 11.4.12]{CLS11}). Since for every $d>0$ we have $\P(a_0,da_1,\ldots,da_n)\cong \P(a_0,a_1,\ldots, a_n)$, we can assume without loss of generality that $\gcd(a_0,\ldots\widehat{a}_{i},\ldots,a_n)=1$ for every $i\in \{0,1,\ldots,n\}$, in which case one says that the description of the weighted projective space is \emph{well-formed}. The singular locus of a well-formed $\P (\bar a)$ can be described as follows \cite[5.15]{Fl00}: 
\begin{equation}\label{eq:SingP}
[x_0: \ldots :x_n]  \in \Sing \P(\bar a) \iff \gcd \{a_i:x_i\neq 0\}>1.
\end{equation}
By \cite[Proposition 2.3]{Mo75} the class group of  $\P(\bar a)$ is isomorphic to $\Z$ and is generated by the class of the divisorial sheaf  $\cO_{\P(\bar a)}(1)$ on $\P(\bar a)$.  
Furthermore, the sheaf $\cO_{\P(\bar a)}(m)$, where $m$ is the least common multiple of $a_0,\ldots, a_n$, is invertible and its class generates the Picard group of $\P(\bar a)$; see also \cite[Exercise 4.1.5 and 4.2.11]{CLS11}.

Put $\bar x=(x_0,\ldots,x_n)$. After fixing $\bar a$ such that the description of $\P(\bar a)$ is well-formed, we denote by $\C[\bar x]_{(d)}$ the set of weighted homogeneous polynomials of degree $d$ in the variables $\bar x$ with respect to the weights $\bar a$. Let $f(\bar x)\in \C[\bar x]_{(d)}$. A zero scheme $Y=Z(f(\bar x))$ has, by definition, \emph{degree} $d$. We say that $Y$ is \emph{quasi-smooth} if its affine cone $$\cC(Y)=\Spec(\C[\bar x]/(f(\bar x)))\subseteq \A^{n+1}=\Spec(\C[\bar x])$$ is smooth off the origin, equivalently, if $\{\pi^*f=0\}$ is a smooth subvariety of $\P^n$, see \eqref{eq:pi}. The singularities of a quasi-smooth subvariety are finite quotient singularities, hence, in particular klt by \cite[Theorem 7.4.9]{Ishii-Intro_to_singularities}. 

\medskip
From now on we assume that $n\geq 4$ and that $a_0=1$. We put 
\begin{equation}\label{eq:P_H}
\P=\P(1,a_1,\ldots,a_n) \quad \text{and}\quad H=\{ x_0=0 \} \cong \P (a_1, \cdots , a_n).
\end{equation}
Then $ \mathcal{O}_\P(H)\cong \mathcal{O}_\P(1)$ generates $\Cl(\P)$ and $\P\setminus H$ is isomorphic to the affine $n$-space $\A^n$ with inhomogeneous coordinates $x_i/x_0^{a_i}$, where $i=1,\ldots,n$. Furthermore, since $K_\P\sim -\sum_{i=0}^n\{x_i=0\}$ (see e.g.\ \cite[\S 2.1]{Dolg-WPS}), the Fano index $i_\P$ of $\P$ is equal to $1+\sum_{i=1}^{n} a_i$. The induced description of $H$ as $\P(a_1,\ldots,a_n)$ is not necessarily well-formed (take for instance $\bar a=(1,1,d,\ldots,d)$), but if it is, then by \eqref{eq:SingP} we have $\Sing\P=\Sing H$, so in this case the singular locus of $\P$ has codimension at least $3$.

\medskip
Theorem \ref{thm1} is a consequence of the combination of the following result with Corollary \ref{cor:MFS-Completion}.

\begin{prop}[Mori fiber completions from pencils on $\P(1,\bar a)$]\label{prop:ConePencilGoodres}
Let $\P$ and $H$ be as in \eqref{eq:P_H}. Assume that $\P$ is smooth in codimension $2$ (equivalently, the induced description of $H$ is well-formed) and let $F\subseteq \P$ be a quasi-smooth terminal hypersurface of degree $d\in \{2,\ldots,\sum_{i=1}^na_i\}$. Then the pencil $\caL$ generated by $F$ and $dH$  is a terminal rank one Fano pencil with quasi-smooth general members and the associated rational map $\psi_\caL\:\P\map \P^1$ admits a compatible thrifty resolution with discrepancy locus contained in $\{(\psi_\caL)_*H\}$.
\end{prop}

\begin{proof}
Let $f\in \C[\bar x]_{(d)}$ be the irreducible weighted homogeneous polynomial defining the hypersurface $F$. The base locus $\Bs \caL$ is the codimension $2$ weighted complete intersection of $\P$ with weighted homogeneous ideal $(f(\bar x),x_0^d)$. The graph $\Gamma$ of $\psi_\caL$ is isomorphic to the hypersurface in $\P\times \P^1_{[u_0:u_1]}$ defined by the bi-homogeneous equation $f(\bar x)u_1+x_0^du_0=0$ and the projection $\pr_{\P}$ induces isomorphisms between closed fibers of the restriction to $\Gamma$ of the projection $\pp=(\pr_{\P^1})|_\Gamma$ and members of $\caL$. 

Since $F$ is terminal, $\caL$ is a terminal pencil by Lemma \ref{lem:Q-term-crit}. By the Lefschetz hyperplane section theorem for weighted projective spaces \cite[Theorem 3.7]{Mo75} and \cite[Remark 4.2]{Okada-stable_rationality_of_Fano_3fold_hypersurfaces}, the group  $\Cl(F)$ is isomorphic to $\Z$. It is generated by the restriction $H|_F$ of $H$ to $F$ as a Weil divisor, for which $\cO_F(H|_F)\cong (\cO_\P(1)|_F)^{\vee\vee}$. This implies in particular that $F\cap H =S$ is irreducible, hence that the exceptional locus $E\cong S\times \P^1$ of $\pr_{\P}\:\Gamma\to \P$ is a prime divisor on $\Gamma$. Since $\P$ is smooth in codimension two, every Weil divisor on it is Cartier in codimension two in $\P$. By assumption $F$ is terminal, so it is smooth in codimension $2$, hence smooth and Cartier at general points of $S$. We have $d\geq 2$, so every member of $\caL$ other than $dH$ is prime by Lemma \ref{lem:integral-members}, and since $2\leq d<i_\P$, every normal member of $\caL$ is Fano by Lemma \ref{lem:fat-member-take0}. Since members of $\caL$ are not necessarily Cartier but only $\Q$-Cartier, the fact that all of them have class group isomorphic to $\Z$, hence have Picard rank one, follows again from the Lefschetz hyperplane section theorem for weighted projective spaces.
Thus, $\caL$ is a terminal rank one Fano pencil with degeneracy locus $\delta(\caL)=\{(\psi_\caL)_*H\}$. 

The affine cone  $\cC(Y)$ over a member $Y$ of $\caL$ other than $dH$ is isomorphic to the hypersurface in $\A^{n+1}$ defined by the equation $f(\bar x)+tx_0^d=0$ for some $t\in \C$. Since $\cC(F)$ is smooth off the origin $\{O\}$, it follows from Bertini's theorem that a general $\cC(Y)$ is smooth outside its intersection with  $\cC(H)=\{x_0=0\}$. 

Furthermore,  since $d\geq 2$, it follows from Lemma \ref{lem:pencil-local-smooth} that $\cC(Y)\setminus \{O\}$ is smooth in a neighborhood of $(\cC(Y)\cap \{x_0=0\})\setminus \{ O \}$. A general member $Y$ is thus quasi-smooth and every member $Y$ other than $dH$ is klt in a neighborhood of $S$.

In view of Lemma \ref{lem:terminalization}, every thrifty $\Q$-factorial terminal resolution of $\caL$ is compatible, provided that the open subset $\Gamma_\infty=\Gamma\setminus H'$ of $\Gamma$, where $H'$ is the proper transform of $H$, is terminal and $\Q$-factorial. Thus it remains to show that $\Gamma_\infty$ is $\Q$-factorial terminal. Every member $Y$ of $\caL$ other than $dH$ is klt in a neighborhood of $S$, so by Proposition \ref{lem:neighborhood-control}(a), $\Gamma_\infty$ is normal in a neighborhood of $E_\infty=E\cap \Gamma_\infty$. We have  $\Gamma_\infty \setminus E_\infty \cong \P \setminus H \cong \A^n$, so $\Gamma_\infty$ is normal and its class group is generated by irreducible components of $E_\infty$. Since $E_\infty$ is irreducible and $\Q$-Cartier, we conclude that $\Gamma_\infty$ is normal and $\Q$-factorial. The pencil $\caL$ is terminal and its members other than $dH$ are klt in a neighborhood of $S$. By Proposition \ref{lem:neighborhood-control}(b), $\Gamma_\infty$  is terminal in a neighborhood of $E_\infty$, hence it is terminal. 
\end{proof}

As a corollary we obtain the following result (cf.\ Definition \ref{dfn:MFS}):

\begin{cor}[Mori fiber completions of $\A^4$ with birationally rigid fibers]\label{cor:95families_bir_rigid_fibers} 
There exist at least $95$ pairwise non weakly square birationally equivalent Mori fiber completions of $\A^4$ over $\P^1$ with quasi-smooth terminal birationally rigid general fibers. 
\end{cor}

\begin{proof} Let $\bar a_j=(a_1^{(j)},a_2^{(j)},a_3^{(j)},a_4^{(j)})$ for $j=1,2$ be two distinct sequences in the list of 95 non-decreasing sequences of \cite[\S 13.3, Lemma 16.4 and \S 16.6]{Fl00}. Put $\P_j =\P(1,\bar a_j)$ and let $d_j=\sum_{i=1}^4 a_i^{(j)}$. Looking at the list one checks that the hyperplane $H_j=\P(\bar a_j)\subset \P_j$ is a well-formed weighted projective space. Let $F_j\subseteq \P_j$ be a general hypersurface of degree $d_j$. By construction $F_j$ is a quasi-smooth and terminal Fano variety of index $1$ anticanonically embedded into $\P_j$. The intersection of $F_j$ with $H_j$ generates the class group of $F_j$, so it is irreducible. General members of the pencil $\caL_j$ on $\P_j$ generated by $F_j$ and $dH_j$ are then quasi-smooth terminal hypersurfaces of $\P_j$ of degree $d_j$, too. By \cite[Theorem 1.1.10]{CP16} they are all birationally rigid.  

By Theorem \ref{thm1} general members of the pencil $\caL_j$ on $\P_j$ are realized as general fibers of a Mori fiber space $p_j\: V_j \to \P^1$, which contains $\A^4\cong \P_j\setminus H_j$. Assume that two Mori fiber spaces for $j=1,2$ are weakly square birational equivalent. Then there exists a birational map $\chi\: V_1 \map V_2$ and an isomorphism $\varphi\: \P^1 \to \P^1$ of the base curve such that $p_2\circ \chi = \varphi \circ p_1$. Then for a general point $t\in \P^1$, $\chi$ induces a birational map between $(V_1)_t=p_1^{-1}(t)$ and $(V_2)_{\varphi(t)}=p_2^{-1}(\varphi(t))$. Since for a general $t$ these threefolds are isomorphic to general members of $\caL_1$ and $\caL_2$ respectively, which are birationally rigid, it follows that general members of $\caL_1$ are isomorphic to general members of $\caL_2$. In particular, the self-intersections of their respective anti-canonical divisors are equal and their singularities are the same. By \cite[\S 16.6]{Fl00} this implies that $\bar a_1=\bar a_2$.
\end{proof} 

Finally, the following example gives a proof of Corollary \ref{cor:completions_of_A4_Okada}.
 
\begin{example}\label{okada} 
Let $n=4$ and let the quadruple $\bar a=(a_1, a_2, a_3,a_4)$ be one of those with numbers:
\[{\rm No.} \,  97-102, 107-110, 116, 117\]
in \cite[Table 1]{Okada-stable_rationality_of_Fano_3fold_hypersurfaces}. Then a very general quasi-smooth hypersurface $F \subset \P(1,\bar a)$ of degree $d=\sum_{j=1}^4 a_j - \alpha$ is a $\Q$-factorial terminal Fano threefold that is not stably rational if $\alpha =3$ for No.\ $107 - 110$ or if $\alpha =5$ for No. 116, 117 or if $\alpha =2$ otherwise. Fixing any such quadruple and choosing a very general hypersurface $F$ of indicated degree $d$ we obtain by Theorem \ref{thm1} a Mori fiber completion of $\A^4$, $\pi\: V \to \P^1$, whose fibers are isomorphic to the members of the pencil $\langle F, dH\rangle$ other than $dH$, hence whose very general fibers are not stably rational. 
\end{example} 

\bibliographystyle{amsalpha} 
\bibliography{bibl_DKP.bib}

\providecommand{\bysame}{\leavevmode\hbox to3em{\hrulefill}\thinspace}
\providecommand{\MR}{\relax\ifhmode\unskip\space\fi MR }
% \MRhref is called by the amsart/book/proc definition of \MR.
\providecommand{\MRhref}[2]{%
  \href{http://www.ams.org/mathscinet-getitem?mr=#1}{#2}
}
\providecommand{\href}[2]{#2}
\begin{thebibliography}{BCHM10}

\bibitem[BCHM10]{BCHM}
Caucher Birkar, Paolo Cascini, Christopher~D. Hacon, and James McKernan,
  \emph{Existence of minimal models for varieties of log general type}, J.
  Amer. Math. Soc. \textbf{23} (2010), no.~2, 405--468.

\bibitem[Bea77]{Beau77}
Arnaud Beauville, \emph{Vari\'{e}t\'{e}s de {P}rym et jacobiennes
  interm\'{e}diaires}, Ann. Sci. \'{E}cole Norm. Sup. (4) \textbf{10} (1977),
  no.~3, 309--391.

\bibitem[BS13]{Schwede-reflexive}
Manuel Blickle and Karl Schwede, \emph{{$p^{-1}$}-linear maps in algebra and
  geometry}, Commutative algebra, Springer, New York, 2013, pp.~123--205.

\bibitem[CDP18]{CDP17}
Ivan Cheltsov, Adrien Dubouloz, and Jihun Park, \emph{Super-rigid affine {F}ano
  varieties}, Compos. Math. \textbf{154} (2018), no.~11, 2462--2484.

\bibitem[CFST16]{CFST-Fano_fibers}
Giulio Codogni, Andrea Fanelli, Roberto Svaldi, and Luca Tasin, \emph{Fano
  varieties in {M}ori fibre spaces}, Int. Math. Res. Not. IMRN (2016), no.~7,
  2026--2067.

\bibitem[CFST18]{CFST-Fano_fibers2}
\bysame, \emph{A note on the fibres of {M}ori fibre spaces}, Eur. J. Math.
  \textbf{4} (2018), no.~3, 859--878.

\bibitem[CG72]{CM72}
C.~Herbert Clemens and Phillip~A. Griffiths, \emph{The intermediate {J}acobian
  of the cubic threefold}, Ann. of Math. (2) \textbf{95} (1972), 281--356.

\bibitem[Che05]{Cheltsov-rigid_Fano}
I.~A. Cheltsov, \emph{Birationally rigid {F}ano varieties}, Uspekhi Mat. Nauk
  \textbf{60} (2005), no.~5(365), 71--160.

\bibitem[CLS11]{CLS11}
David~A. Cox, John~B. Little, and Henry~K. Schenck, \emph{Toric varieties},
  Graduate Studies in Mathematics, vol. 124, American Mathematical Society,
  Providence, RI, 2011.

\bibitem[Cor00]{Corti-Sing_of_lin_sys}
Alessio Corti, \emph{Singularities of linear systems and {$3$}-fold birational
  geometry}, Explicit birational geometry of 3-folds, London Math. Soc. Lecture
  Note Ser., vol. 281, Cambridge Univ. Press, Cambridge, 2000, pp.~259--312.

\bibitem[CP17]{CP16}
Ivan Cheltsov and Jihun Park, \emph{Birationally rigid {F}ano threefold
  hypersurfaces}, Mem. Amer. Math. Soc. \textbf{246} (2017), no.~1167, v+117.

\bibitem[CPR00]{CPR00}
Alessio Corti, Aleksandr Pukhlikov, and Miles Reid, \emph{Fano {$3$}-fold
  hypersurfaces}, Explicit birational geometry of 3-folds, London Math. Soc.
  Lecture Note Ser., vol. 281, Cambridge Univ. Press, Cambridge, 2000,
  pp.~175--258.

\bibitem[dF13]{deF13}
Tommaso de~Fernex, \emph{Birationally rigid hypersurfaces}, Invent. Math.
  \textbf{192} (2013), no.~3, 533--566.

\bibitem[DK17]{DK3}
Adrien Dubouloz and Takashi Kishimoto, \emph{Explicit biregular/birational
  geometry of affine threefolds: completions of {$\Bbb A^3$} into del {P}ezzo
  fibrations and {M}ori conic bundles}, Algebraic varieties and automorphism
  groups, Adv. Stud. Pure Math., vol.~75, Math. Soc. Japan, Tokyo, 2017,
  pp.~49--71.

\bibitem[DK18]{DK4}
\bysame, \emph{Cylinders in del {P}ezzo fibrations}, Israel J. Math.
  \textbf{225} (2018), no.~2, 797--815.

\bibitem[Dol82]{Dolg-WPS}
Igor Dolgachev, \emph{Weighted projective varieties}, Group actions and vector
  fields ({V}ancouver, {B}.{C}., 1981), Lecture Notes in Math., vol. 956,
  Springer, Berlin, 1982, pp.~34--71.

\bibitem[Dol12]{Dolgachev-Classical_AG_modern_view}
Igor~V. Dolgachev, \emph{Classical algebraic geometry}, Cambridge University
  Press, Cambridge, 2012, A modern view.

\bibitem[FA92]{FA}
\emph{Flips and abundance for algebraic threefolds}, Soci\'et\'e Math\'ematique
  de France, Paris, 1992, Papers from the Second Summer Seminar on Algebraic
  Geometry held at the University of Utah, Salt Lake City, Utah, August 1991,
  Ast{\'e}risque No. 211 (1992).

\bibitem[Fuj81]{Fujita81}
Takao Fujita, \emph{On the structure of polarized manifolds with total
  deficiency one. {II}}, J. Math. Soc. Japan \textbf{33} (1981), no.~3,
  415--434.

\bibitem[Fuk19]{Fukuoka19}
Takeru Fukuoka, \emph{Refinement of the classification of weak {F}ano
  threefolds with sextic del {P}ezzo fibrations}, \arxiv{1903.06872}, 2019.

\bibitem[Ful98]{Fulton-Intersection_theory}
William Fulton, \emph{Intersection theory}, second ed., Ergebnisse der
  Mathematik und ihrer Grenzgebiete. 3. Folge. A Series of Modern Surveys in
  Mathematics [Results in Mathematics and Related Areas. 3rd Series. A Series
  of Modern Surveys in Mathematics], vol.~2, Springer-Verlag, Berlin, 1998.

\bibitem[Fur86]{Fu86}
Mikio Furushima, \emph{Singular del {P}ezzo surfaces and analytic
  compactifications of {$3$}-dimensional complex affine space {${\bf C}^3$}},
  Nagoya Math. J. \textbf{104} (1986), 1--28.

\bibitem[Fur90]{H3b}
\bysame, \emph{Complex analytic compactifications of {${\bf C}^3$}}, vol.~76,
  1990, Algebraic geometry (Berlin, 1988), pp.~163--196.

\bibitem[Fur93]{Fu93}
\bysame, \emph{The complete classification of compactifications of {${\bf
  C}^3$} which are projective manifolds with the second {B}etti number one},
  Math. Ann. \textbf{297} (1993), no.~4, 627--662.

\bibitem[Gro61]{EGAIII-1}
A.~Grothendieck, \emph{\'{E}l\'{e}ments de g\'{e}om\'{e}trie alg\'{e}brique.
  {III}. \'{E}tude cohomologique des faisceaux coh\'{e}rents. {I}}, Inst.
  Hautes \'{E}tudes Sci. Publ. Math. (1961), no.~11, 167.

\bibitem[Gro65]{EGAIV-2}
\bysame, \emph{\'{E}l\'{e}ments de g\'{e}om\'{e}trie alg\'{e}brique. {IV}.
  \'{E}tude locale des sch\'{e}mas et des morphismes de sch\'{e}mas. {II}},
  Inst. Hautes \'{E}tudes Sci. Publ. Math. (1965), no.~24, 231.

\bibitem[Gro67]{EGAIV-4}
\bysame, \emph{\'{E}l\'{e}ments de g\'{e}om\'{e}trie alg\'{e}brique. {IV}.
  \'{E}tude locale des sch\'{e}mas et des morphismes de sch\'{e}mas {IV}},
  Inst. Hautes \'{E}tudes Sci. Publ. Math. (1967), no.~32, 361.

\bibitem[Gro68]{SGA2}
Alexander Grothendieck, \emph{Cohomologie locale des faisceaux coh\'{e}rents et
  th\'{e}or\`emes de {L}efschetz locaux et globaux {$(SGA$} {$2)$}},
  North-Holland Publishing Co., Amsterdam; Masson \& Cie, \'{E}diteur, Paris,
  1968, Augment\'{e} d'un expos\'{e} par Mich\`ele Raynaud, S\'{e}minaire de
  G\'{e}om\'{e}trie Alg\'{e}brique du Bois-Marie, 1962, Advanced Studies in
  Pure Mathematics, Vol. 2.

\bibitem[Har77]{Hartshorne}
Robin Hartshorne, \emph{Algebraic geometry}, Springer-Verlag, New
  York-Heidelberg, 1977, Graduate Texts in Mathematics, No. 52.

\bibitem[Hir54]{Hirzebruch_problems}
Friedrich Hirzebruch, \emph{Some problems on differentiable and complex
  manifolds}, Ann. of Math. (2) \textbf{60} (1954), 213--236.

\bibitem[IF00]{Fl00}
A.~R. Iano-Fletcher, \emph{Working with weighted complete intersections},
  Explicit birational geometry of 3-folds, London Math. Soc. Lecture Note Ser.,
  vol. 281, Cambridge Univ. Press, Cambridge, 2000, pp.~101--173.

\bibitem[IM71]{IsMa71}
V.~A. Iskovskih and Ju.~I. Manin, \emph{Three-dimensional quartics and
  counterexamples to the {L}\"{u}roth problem}, Mat. Sb. (N.S.)
  \textbf{86(128)} (1971), 140--166.

\bibitem[IP99]{Iskovskikh_Prokhorov-Fano_varieties}
V.~A. Iskovskikh and Yu.~G. Prokhorov, \emph{Fano varieties}, Algebraic
  geometry, {V}, Encyclopaedia Math. Sci., vol.~47, Springer, Berlin, 1999,
  pp.~1--247.

\bibitem[Ish18]{Ishii-Intro_to_singularities}
Shihoko Ishii, \emph{Introduction to singularities}, Springer, Tokyo, 2018,
  Second edition.

\bibitem[Kas13]{Kasprzyk-terminal_WPS}
Alexander Kasprzyk, \emph{Classifying terminal weighted projective space},
  \arxiv{1304.3029}, 2013.

\bibitem[Kaw88]{Kawamata-Crepant-Bl-3d}
Yujiro Kawamata, \emph{Crepant blowing-up of {$3$}-dimensional canonical
  singularities and its application to degenerations of surfaces}, Ann. of
  Math. (2) \textbf{127} (1988), no.~1, 93--163.

\bibitem[Kis05]{Kis05}
Takashi Kishimoto, \emph{Compactifications of contractible affine 3-folds into
  smooth {F}ano 3-folds with {$B_2=2$}}, Math. Z. \textbf{251} (2005), no.~4,
  783--820.

\bibitem[KM92]{KollarMori-3d_flips}
J\'{a}nos Koll\'{a}r and Shigefumi Mori, \emph{Classification of
  three-dimensional flips}, J. Amer. Math. Soc. \textbf{5} (1992), no.~3,
  533--703.

\bibitem[KM98]{KollarMori-Bir_geometry}
J{\'a}nos Koll{\'a}r and Shigefumi Mori, \emph{Birational geometry of algebraic
  varieties}, Cambridge Tracts in Mathematics, vol. 134, Cambridge University
  Press, Cambridge, 1998, With the collaboration of C. H. Clemens and A. Corti,
  Translated from the 1998 Japanese original.

\bibitem[Kol96]{Kollar-Rational_curves_on_alg_var}
J\'{a}nos Koll\'{a}r, \emph{Rational curves on algebraic varieties}, Ergebnisse
  der Mathematik und ihrer Grenzgebiete. 3. Folge. A Series of Modern Surveys
  in Mathematics [Results in Mathematics and Related Areas. 3rd Series. A
  Series of Modern Surveys in Mathematics], vol.~32, Springer-Verlag, Berlin,
  1996.

\bibitem[Kol97]{Kollar-Singularities_of_pairs}
\bysame, \emph{Singularities of pairs}, Algebraic geometry---{S}anta {C}ruz
  1995, Proc. Sympos. Pure Math., vol.~62, Amer. Math. Soc., Providence, RI,
  1997, pp.~221--287.

\bibitem[Kol13]{Kollar-Singularities_of_MMP}
\bysame, \emph{Singularities of the minimal model program}, Cambridge Tracts in
  Mathematics, vol. 200, Cambridge University Press, Cambridge, 2013, With a
  collaboration of S\'{a}ndor Kov\'{a}cs.

\bibitem[KSC04]{Kollar-Smith-Corti_Nearly_rational}
J\'{a}nos Koll\'{a}r, Karen~E. Smith, and Alessio Corti, \emph{Rational and
  nearly rational varieties}, Cambridge Studies in Advanced Mathematics,
  vol.~92, Cambridge University Press, Cambridge, 2004.

\bibitem[Man86]{Manin-cubic_forms}
Yu.~I. Manin, \emph{Cubic forms}, second ed., North-Holland Mathematical
  Library, vol.~4, North-Holland Publishing Co., Amsterdam, 1986, Algebra,
  geometry, arithmetic, Translated from the Russian by M. Hazewinkel.

\bibitem[Mat89]{Matsumura}
Hideyuki Matsumura, \emph{Commutative ring theory}, second ed., Cambridge
  Studies in Advanced Mathematics, vol.~8, Cambridge University Press,
  Cambridge, 1989, Translated from the Japanese by M. Reid.

\bibitem[Mat02]{Matsuki}
Kenji Matsuki, \emph{Introduction to the {M}ori program}, Universitext,
  Springer-Verlag, New York, 2002.

\bibitem[Mor75]{Mo75}
Shigefumi Mori, \emph{On a generalization of complete intersections}, J. Math.
  Kyoto Univ. \textbf{15} (1975), no.~3, 619--646.

\bibitem[Oka19]{Okada-stable_rationality_of_Fano_3fold_hypersurfaces}
Takuzo Okada, \emph{Stable rationality of orbifold {F}ano 3-fold
  hypersurfaces}, J. Algebraic Geom. \textbf{28} (2019), no.~1, 99--138.

\bibitem[Pro91]{Pr}
Yu.~G. Prokhorov, \emph{Fano threefolds of genus {$12$} and compactifications
  of {${\bf C}^3$}}, Algebra i Analiz \textbf{3} (1991), no.~4, 162--170.

\bibitem[Pro94]{Pr93}
Yuri~G. Prokhorov, \emph{Compactifications of {${\bf C}^4$} of index {$3$}},
  Algebraic geometry and its applications ({Y}aroslavl', 1992), Aspects Math.,
  E25, Friedr. Vieweg, Braunschweig, 1994, pp.~159--169.

\bibitem[Pro16]{Pr16}
Yu.~G. Prokhorov, \emph{Singular {F}ano manifolds of genus 12}, Mat. Sb.
  \textbf{207} (2016), no.~7, 101--130.

\bibitem[PS88]{H3a}
Thomas Peternell and Michael Schneider, \emph{Compactifications of {${\bf
  C}^3$}. {I}}, Math. Ann. \textbf{280} (1988), no.~1, 129--146.

\bibitem[Puk98]{Pu98}
Aleksandr~V. Pukhlikov, \emph{Birational automorphisms of {F}ano
  hypersurfaces}, Invent. Math. \textbf{134} (1998), no.~2, 401--426.

\bibitem[Puk13]{Pukhlikov-bir_rigid}
Aleksandr Pukhlikov, \emph{Birationally rigid varieties}, Mathematical Surveys
  and Monographs, vol. 190, American Mathematical Society, Providence, RI,
  2013.

\bibitem[PZ18]{PZ-genus10_4folds_and_A4}
Yuri Prokhorov and Mikhail Zaidenberg, \emph{Fano-{M}ukai fourfolds of genus 10
  as compactifications of {$\Bbb C^4$}}, Eur. J. Math. \textbf{4} (2018),
  no.~3, 1197--1263.

\bibitem[Rei80]{Reid-canonical_3folds}
Miles Reid, \emph{Canonical {$3$}-folds}, Journ\'{e}es de {G}\'{e}ometrie
  {A}lg\'{e}brique d'{A}ngers, {J}uillet 1979/{A}lgebraic {G}eometry, {A}ngers,
  1979, Sijthoff \& Noordhoff, Alphen aan den Rijn---Germantown, Md., 1980,
  pp.~273--310.

\bibitem[SS85]{Vanishing-Theorems-Cplx-Manifolds}
Bernard Shiffman and Andrew~John Sommese, \emph{Vanishing theorems on complex
  manifolds}, Progress in Mathematics, vol.~56, Birkh\"{a}user Boston, Inc.,
  Boston, MA, 1985.

\bibitem[Tak00]{Takayama-pi_1(lt_Fano)}
Shigeharu Takayama, \emph{Simple connectedness of weak {F}ano varieties}, J.
  Algebraic Geom. \textbf{9} (2000), no.~2, 403--407.

\bibitem[Zha06]{Zhang-Fanos_rationally_connected}
Qi~Zhang, \emph{Rational connectedness of log {${\bf Q}$}-{F}ano varieties}, J.
  Reine Angew. Math. \textbf{590} (2006), 131--142.

\end{thebibliography}

\end{document}